\documentclass[a4paper,11pt,reqno]{amsart}

\usepackage[utf8]{inputenc}

\usepackage{etex}

\usepackage{xcolor}

\definecolor{verydarkblue}{rgb}{0,0,0.5}

\usepackage[
breaklinks,
colorlinks,
citecolor=verydarkblue,
linkcolor=verydarkblue,
urlcolor=verydarkblue,
pagebackref=true,
hyperindex
]{hyperref}

\backrefenglish

\usepackage{fancyhdr}

\usepackage[
hscale=0.7,
vscale=0.75,
headheight=13pt,
marginpar=2cm,
centering,
]{geometry}

\usepackage{amsmath}
\usepackage{amsthm}
\usepackage{amssymb}
\usepackage{mathtools}
\usepackage{mathdots}
\usepackage{framed}
\usepackage[capitalize]{cleveref}
\usepackage{array}
\usepackage[alphabetic]{amsrefs}
\usepackage[all,cmtip]{xy}

\theoremstyle{plain}

\newtheorem{introtheorem}{Theorem}

\crefname{introtheorem}{Theorem}{Theorems}

\newtheorem{theorem}{Theorem}[section]
\newtheorem{proposition}[theorem]{Proposition}
\newtheorem{lemma}[theorem]{Lemma}
\newtheorem{corollary}[theorem]{Corollary}

\newtheorem{question}[theorem]{Question}

\theoremstyle{definition}

\newtheorem{definition}[theorem]{Definition}

\theoremstyle{remark}

\newtheorem{remark}[theorem]{Remark}
\newtheorem{example}[theorem]{Example}

\numberwithin{figure}{section}

\numberwithin{equation}{section}

\IfFileExists{./article-style.tex}{

\pagestyle{fancy}
\fancyhead{}
\fancyfoot{}
\fancyhead[LE,RO]{\small \thepage}
\fancyhead[RE]{\small \nouppercase{\rightmark}}
\fancyhead[LO]{\small \nouppercase{\leftmark}}

\setcounter{tocdepth}{1}

\makeatletter

\global\BR@BackrefAlttrue

\renewcommand*{\backrefalt}[4]{%
    \tiny%
    (%
    \ifcase #1 not cited%
          \or cit.~on~p.~#2%
          \else cit.~on~pp.~#2%
    \fi%
    )%
}

\def\print@backrefs#1{%
    \space\SentenceSpace%
    \begingroup%
        \expandafter\providecommand\csname brc@#1\endcsname{0}%
        \expandafter\providecommand\csname brcd@#1\endcsname{0}%
        \expandafter\backrefalt%
            \csname brc@#1\expandafter\endcsname%
            \csname brl@#1\expandafter\endcsname%
            \csname brcd@#1\expandafter\endcsname%
            \csname brld@#1\endcsname%
    \endgroup%
}

\def\maketitle{\par
  \@topnum\z@ 
  \@setcopyright
  \thispagestyle{empty}
  \ifx\@empty\shortauthors \let\shortauthors\shorttitle
  \else \andify\shortauthors
  \fi
  \@maketitle@hook
  \begingroup
  \@maketitle
  \toks@\@xp{\shortauthors}\@temptokena\@xp{\shorttitle}%
  \toks4{\def\\{ \ignorespaces}}
  \edef\@tempa{%
    \@nx\markboth{\the\toks4
      \@nx\MakeUppercase{\the\toks@}}{\the\@temptokena}}%
  \@tempa
  \endgroup
  \c@footnote\z@
    \renewcommand{\footnoterule}{%
      \kern -3pt
      \hrule width \textwidth height .5pt
      \kern 2pt
    }
  {
    \renewcommand\thefootnote{}
    \vspace{-2em}
    \footnote{
      \par\vspace{-1.2em}\noindent%
      \setlength{\parindent}{0pt}%
      \def\@footnotetext##1{\noindent{\footnotesize##1}\par}%
      \let\@makefnmark\relax  \let\@thefnmark\relax
      \ifx\@empty\@date\else \@footnotetext{\@setdate}\fi
      \ifx\@empty\@subjclass\else \@footnotetext{\@setsubjclass}\fi
      \ifx\@empty\@keywords\else \@footnotetext{\@setkeywords}\fi
      \ifx\@empty\thankses\else \@footnotetext{%
        \@setthanks}%
      \fi
    }
    \addtocounter{footnote}{-1}
  }
  \@cleartopmattertags
}

\def\@adminfootnotes{\@empty}

\def\@settitle{\begin{center}%
  \baselineskip14\p@\relax
    \bfseries
\Large
  \@title
  \end{center}%
}

\def\@setauthors{%
  \begingroup
  \def\thanks{\protect\thanks@warning}%
  \trivlist
  \centering\footnotesize \@topsep30\p@\relax
  \advance\@topsep by -\baselineskip
  \item\relax
  \author@andify\authors
  \def\\{\protect\linebreak}%
  \large{\authors}%
  \ifx\@empty\contribs
  \else
    ,\penalty-3 \space \@setcontribs
    \@closetoccontribs
  \fi
  \endtrivlist
  \endgroup
}

\def\@setaddresses{\par
  \nobreak \begingroup
\footnotesize
  \def\author##1{\end{minipage}\hskip 2em \begin{minipage}[t]{.5\textwidth minus
  1em}\raggedright%
    ~\\[2em]{\bf##1}\\[.5em]%
  }%
  \interlinepenalty\@M
  \def\address##1##2{\begingroup
    {\ignorespaces##2}\endgroup\\[.5em]}%
  \def\curraddr##1##2{\begingroup
    \@ifnotempty{##2}{\nobreak\indent\curraddrname
      \@ifnotempty{##1}{, \ignorespaces##1\unskip}\/:\space
      ##2\par}\endgroup}%
  \def\email##1##2{\begingroup
    \@ifnotempty{##2}{\nobreak\indent
      \@ifnotempty{##1}{, \ignorespaces##1\unskip}
      \ttfamily##2\par}\endgroup}%
  \def\urladdr##1##2{\begingroup
    \def~{\char`\~}%
    \@ifnotempty{##2}{\nobreak\indent\urladdrname
      \@ifnotempty{##1}{, \ignorespaces##1\unskip}\/:\space
      \ttfamily##2\par}\endgroup}%
  \setlength{\parindent}{0pt}%
  \vfill%
  {
  \hskip -2em%
  \begin{minipage}{0mm}
  \addresses
  \end{minipage}
  }
  \endgroup
}

\renewcommand{\author}[2][]{%
  \ifx\@empty\authors
    \gdef\authors{#2}%
    \g@addto@macro\addresses{\author{#2}}%
  \else
    \g@addto@macro\authors{\and#2}%
    \g@addto@macro\addresses{\author{#2}}%
  \fi
  \@ifnotempty{#1}{%
    \ifx\@empty\shortauthors
      \gdef\shortauthors{#1}%
    \else
      \g@addto@macro\shortauthors{\and#1}%
    \fi
  }%
}
\edef\author{\@nx\@dblarg
  \@xp\@nx\csname\string\author\endcsname}

\def\@secnumfont{\@empty}

\def\section{\@startsection{section}{1}%
  \z@{.7\linespacing\@plus\linespacing}{.5\linespacing}%
  {\large\bfseries\centering}}

\DefineAdditiveKey{bib}{secondauthor}{\name}
\DefineSimpleKey{bib}{archivePrefix}
\DefineSimpleKey{bib}{eprinttype}
\DefineSimpleKey{bib}{eprintclass}
\DefineSimpleKey{bib}{primaryClass}

\newcommand{\my@ifeq}[3]{%
    \edef\my@a{#1}%
    \edef\my@b{#2}%
    \ifx\my@a\my@b#3\fi%
}

\newcommand{\@doititle@doi}[1]{%
    \href%
        {https://doi.org/\csname bib'doi\endcsname}%
        {\textit{#1}}%
}

\newcommand{\@doititle@url}[1]{%
    \href%
        {\csname bib'url\endcsname}%
        {\textit{#1}}%
}

\newcommand{\@doititle@mr}[1]{{%
    \def\MR##1{##1}%
    \def\fld@elt##1{##1}%
    \href%
        {http://www.ams.org/mathscinet-getitem?mr=\csname bib'review\endcsname}%
        {\textit{#1}}%
}}

\newcommand{\my@eprint}[1]{%
    \IfEmptyBibField{eprinttype}{%
        \IfEmptyBibField{archivePrefix}{%
            \edef\my@etype{arXiv}%
        }{%
            \edef\my@etype{\csname bib'archivePrefix\endcsname}%
        }%
    }{%
        \edef\my@etype{\csname bib'eprinttype\endcsname}%
    }%
    \def\my@arxiv{%
        \href%
            {https://arxiv.org/abs/#1}%
            {\tt arXiv:#1}%
    }%
    \my@ifeq{\my@etype}{arxiv}{\my@arxiv}%
    \my@ifeq{\my@etype}{arXiv}{\my@arxiv}%
    \my@ifeq{\my@etype}{ArXiv}{\my@arxiv}%
    \def\my@hal{%
        \href%
            {https://hal.archives-ouvertes.fr/#1}%
            {\tt #1}%
    }%
    \my@ifeq{\my@etype}{hal}{\my@hal}%
    \my@ifeq{\my@etype}{HAL}{\my@hal}%
    \my@ifeq{\my@etype}{Hal}{\my@hal}%
}

\newcommand{\@doititle@eprint}[1]{%
    \IfEmptyBibField{eprinttype}{%
        \IfEmptyBibField{archivePrefix}{%
            \edef\my@etype{arXiv}%
        }{%
            \edef\my@etype{\csname bib'archivePrefix\endcsname}%
        }%
    }{%
        \edef\my@etype{\csname bib'eprinttype\endcsname}%
    }%
    \def\my@arxiv{%
        \href%
            {https://arxiv.org/abs/\csname bib'eprint\endcsname}%
            {\textit{#1}}%
    }%
    \my@ifeq{\my@etype}{arxiv}{\my@arxiv}%
    \my@ifeq{\my@etype}{arXiv}{\my@arxiv}%
    \my@ifeq{\my@etype}{ArXiv}{\my@arxiv}%
    \def\my@hal{%
        \href%
            {https://hal.archives-ouvertes.fr/\csname bib'eprint\endcsname}%
            {\textit{#1}}%
    }%
    \my@ifeq{\my@etype}{hal}{\my@hal}%
    \my@ifeq{\my@etype}{HAL}{\my@hal}%
    \my@ifeq{\my@etype}{Hal}{\my@hal}%
}

\newcommand{\@doititle}[1]{%
    \IfEmptyBibField{doi}{%
        \IfEmptyBibField{url}{%
            \IfEmptyBibField{review}{%
                \IfEmptyBibField{eprint}{%
                    \let\@tempa\textit
                }{%
                    \let\@tempa\@doititle@eprint
                }
            }{%
                \let\@tempa\@doititle@mr
            }%
        }{%
            \let\@tempa\@doititle@url
        }%
    }{%
        \let\@tempa\@doititle@doi
    }%
    \@tempa{#1}%
}

\newcommand{\PrintSecondAuthors}[1]{%
    \ifx\previous@primary\current@primary
        \@empty
    \else
        \PrintNames{, with }{}{#1}%
    \fi
}

\BibSpec{article}{%
    +{}  {\PrintAuthors}                {author}
    +{}  {\PrintSecondAuthors}          {secondauthor}
    +{,} { \@doititle}                  {title}
    +{.} { }                            {part}
    +{.} { \@doititle}                  {subtitle}
    +{,} { \PrintContributions}         {contribution}
    +{.} { \PrintPartials}              {partial}
    +{,} { }                            {journal}
    +{}  { \textbf}                     {volume}
    +{}  { \PrintDatePV}                {date}
    +{,} { \issuetext}                  {number}
    +{,} { \eprintpages}                {pages}
    +{,} { }                            {status}
    +{,} { available at \my@eprint}     {eprint}
    +{}  { \PrintTranslation}           {translation}
    +{;} { \PrintReprint}               {reprint}
    +{.} { }                            {note}
    +{.} {}                             {transition}
}

\BibSpec{book}{%
    +{}  {\PrintPrimary}                {transition}
    +{,} { \@doititle}                  {title}
    +{.} { }                            {part}
    +{.} { \@doititle}                  {subtitle}
    +{,} { \PrintEdition}               {edition}
    +{}  { \PrintEditorsB}              {editor}
    +{,} { \PrintTranslatorsC}          {translator}
    +{,} { \PrintContributions}         {contribution}
    +{,} { }                            {series}
    +{,} { \voltext}                    {volume}
    +{,} { }                            {publisher}
    +{,} { }                            {organization}
    +{,} { }                            {address}
    +{,} { \PrintDateB}                 {date}
    +{,} { }                            {status}
    +{}  { \parenthesize}               {language}
    +{}  { \PrintTranslation}           {translation}
    +{;} { \PrintReprint}               {reprint}
    +{.} { }                            {note}
    +{.} {}                             {transition}
}

\BibSpec{collection.article}{%
    +{}  {\PrintAuthors}                {author}
    +{,} { \@doititle}                  {title}
    +{.} { }                            {part}
    +{.} { \@doititle}                  {subtitle}
    +{,} { \PrintContributions}         {contribution}
    +{,} { \PrintConference}            {conference}
    +{}  {\PrintBook}                   {book}
    +{,} { }                            {booktitle}
    +{,} { \PrintDateB}                 {date}
    +{,} { pp.~}                        {pages}
    +{,} { }                            {status}
    +{,} { available at \eprint}        {eprint}
    +{}  { \parenthesize}               {language}
    +{;} { \PrintReprint}               {reprint}
    +{.} { }                            {note}
    +{.} {}                             {transition}
}

\BibSpec{webpage}{%
    +{}  {\PrintAuthors}                {author}
    +{,} { \@doititle}                  {title}
    +{.} { \@doititle}                  {subtitle}
    +{}  { \PrintDate}                  {date}
    +{,} { \url}                        {url}
    +{.} { Accessed \PrintDateField}    {accessdate}
    +{.} { }                            {note}
    +{.} {}                             {transition}
}

\makeatother
}{}

\usepackage{marginnote}

\def\N{{\mathbb N}}
\def\Z{{\mathbb Z}}
\def\Q{{\mathbb Q}}

\def\A{{\mathbb A}}

\def\cA{\mathcal{A}}

\def\cO{\mathcal{O}}

\def\cX{\mathcal{X}}

\def\I{\mathcal{I}}

\def\O{\mathcal{O}}

\def\fm{\mathfrak{m}}
\def\fn{\mathfrak{n}}
\def\fp{\mathfrak{p}}
\def\fq{\mathfrak{q}}

\def\a{\alpha}
\def\b{\beta}
\def\g{\gamma}

\def\f{\varphi}

\def\e{\eta}

\def\k{k}
\def\n{\nu}
\def\m{\mu}

\def\D{\Delta}

\def\Om{\Omega}
\def\vp{\varphi}

\def\.{\cdot}
\let\circum\^
\def\^{\widehat}
\def\o{\circ}
\def\ov{\overline}

\def\inj{\hookrightarrow}

\def\({\left(}
\def\){\right)}

\def\liminv{\underleftarrow{\lim}}

\def\*{{}^*}

\renewcommand{\and}{ \ \ \text{ and } \ \ }

\def\red{\mathrm{red}}
\def\an{\mathrm{an}}
\def\sm{\mathrm{sm}}

\def\Jac{\mathrm{Jac}}

\DeclareMathOperator{\codim} {codim}

\DeclareMathOperator{\im} {Im}

\DeclareMathOperator{\GL}  {GL}

\DeclareMathOperator{\Gr} {Gr}

\DeclareMathOperator{\Spec} {Spec}
\DeclareMathOperator{\Spf} {Spf}

\DeclareMathOperator{\Sing} {Sing}

\DeclareMathOperator{\Val} {Val}

\DeclareMathOperator{\ord} {ord}

\DeclareMathOperator{\Sym} {Sym}
\DeclareMathOperator{\Supp} {Supp}

\DeclareMathOperator{\id} {id}
\DeclareMathOperator{\Fitt} {Fitt}

\DeclareMathOperator{\Hom} {Hom}

\DeclareMathOperator{\gr} {gr}

\DeclareMathOperator{\Tor} {Tor}

\DeclareMathOperator{\height} {ht}

\DeclareMathOperator{\Quot} {Quot}

\DeclareMathOperator{\Nil}{Nil}
\DeclareMathOperator{\MJ} {MJ}
\DeclareMathOperator{\DivVal} {DivVal}
\DeclareMathOperator{\RZ} {RZ}
\DeclareMathOperator{\HS} {HS}

\DeclareMathOperator{\charK}{char}

\def\embdim{\mathrm{edim}}
\def\embcodim{\mathrm{ecodim}}

\def\sep{{\mathrm{sep}}}

\DeclareMathOperator{\trdeg}{tr.deg}

\def\isom{\simeq}

\DeclareMathOperator{\sepdeg}{sep.deg}

\def\cotimes{\hat\otimes}

\newcommand{\LS}[1]{#1(\hskip-1pt(t)\hskip-1pt)}
\newcommand{\LSS}[2]{#1(\hskip-1pt(#2)\hskip-1pt)}

\usepackage[]{soul}

\begin{document}



\title{On arc fibers of morphisms of schemes}

\author{Christopher Chiu}

\address[C.\ Chiu]{%
Department of Mathematics and Computer Science\\
Eindhoven University of Technology\\
De Groene Loper 5\\
5612 AZ Eindhoven (Netherlands)%
}

\email{c.h.chiu@tue.nl}

\author{Tommaso de Fernex}

\address[T.\ de Fernex]{%
Department of Mathematics\\
University of Utah\\
155 South 1400 East\\
Salt Lake City, UT 84112 (USA)%
}

\email{defernex@math.utah.edu}

\author{Roi Docampo}

\address[R.\ Docampo]{%
Department of Mathematics\\
University of Oklahoma\\
601 Elm Avenue, Room 423\\
Norman, OK 73019 (USA)%
}

\email{roi@ou.edu}

\subjclass[2020]{%
Primary {\scriptsize 14E18};
Secondary {\scriptsize 14B25}.}
\keywords{Arc space, generically finite morphism, embedding dimension, stable point.}

\thanks{%
The research of the first author was partially supported by the NWO Vici grant 639.033.514.
The research of the second author was partially supported by NSF Grant 
DMS-2001254.
The research of the third author was partially supported by a grant from the
Simons Foundation (638459,~RD)%
}

\begin{abstract}
Given a morphism $f \colon X \to Y$ of schemes over a field, 
we prove several finiteness results about the fibers of 
the induced map on arc spaces $f_\infty \colon X_\infty \to Y_\infty$.
Assuming that $f$ is quasi-finite and $X$ is separated and quasi-compact, 
our theorem states that $f_\infty$ has topologically finite fibers of bounded cardinality
and its restriction to $X_\infty \setminus R_\infty$, 
where $R$ is the ramification locus of $f$, has scheme-theoretically finite reduced fibers.
We also provide an effective bound on the cardinality of the fibers of $f_\infty$
when $f$ is a finite morphism of varieties over an algebraically closed field, 
describe the ramification locus of $f_\infty$, 
and prove a general criterion for $f_\infty$ to be a morphism of finite type. 
We apply these results to further explore the local structure of arc spaces. 
One application is that the local ring at a stable point of the arc space of a variety
has finitely generated maximal ideal and topologically Noetherian
spectrum, something that should be contrasted with the fact that these rings are not Noetherian in general;
a lower-bound to the dimension of these rings is also obtained.  
Another application gives a semicontinuity property for 
the embedding dimension and embedding codimension of arc spaces which
extends to this setting a theorem of Lech on Noetherian local rings
and translates into a semicontinuity property for Mather log discrepancies. 
Other applications are discussed in the paper.
\end{abstract}

\maketitle

\section{Introduction}

\subsection{Results on arc fibers}

For every morphism $f \colon X \to Y$ of schemes over a field $k$, 
there is a diagram
\[
\xymatrix{
X_\infty \ar[r]^{f_\infty} \ar[d] & Y_\infty \ar[d] \\
X \ar[r]^f & Y
}
\]
where $X_\infty$ and $Y_\infty$ are the respective arc spaces over $k$. 
We use the term \emph{arc fibers} to refer to the (scheme-theoretic) fibers of $f_\infty$. 
%

One of the first properties one learns about arc spaces is that
if $f$ is \'etale then the diagram is Cartesian and $f_\infty$ is \'etale.
What other properties of $f$ lift to $f_\infty$? 

This question is the starting point of the paper. 
Our answer takes a particularly clean form when we restrict the attention to quasi-finite morphisms.

\begin{introtheorem}[\cref{t:finite-fiber}]
\label{t:intro:finite-fiber}
Let $f \colon  X\to Y$ be a quasi-finite morphism of schemes over a perfect field $k$, 
and assume that $X$ is separated and quasi-compact.
Let $R := \Supp \Om_{X/Y}$ denote the ramification locus of $f$. 
Then the induced morphism $f_\infty\colon X_\infty \to Y_\infty$ satisfies the following properties:
\begin{enumerate}
\item 
\label{i1:intro:finite-fiber}
$f_\infty$ has topologically finite fibers of bounded cardinality. 
\item
\label{i2:intro:finite-fiber}
The restriction of $f_\infty$ 
to $X_\infty \setminus R_\infty$ has finite reduced fibers.
\end{enumerate}
\end{introtheorem}

If we further restrict the attention to finite morphisms of varieties and 
assume that the ground field is algebraically closed, then we obtain
an effective bound on the cardinality of the arc fibers
in terms of the separable degree of $f$. 

\begin{introtheorem}[\cref{t:finite-fiber-deg}]
\label{t:intro:finite-fiber-deg}
Let $f \colon X \to Y$ be a finite surjective morphism between varieties over an algebraically closed field $k$, 
and assume that $Y$ is normal. Then $|f^{-1}_\infty(\beta)| \leq \sepdeg f$ for every $\b \in Y_\infty$.
\end{introtheorem}

One should compare these results to the basic fact that if $f$ is proper and birational 
with exceptional locus $E$ then $f_\infty$ induces a bijection from $X_\infty \setminus E_\infty$ to $Y_\infty \setminus (f(E))_\infty$, a property that follows directly from the valuative criterion of properness. 

\cref{t:intro:finite-fiber} is obtained from more precise results holding 
under weaker assumptions on $f$, which can be useful in applications (see  
\cref{t:finite-fiber-1,t:finite-fiber-2}). The proof broadly involves three different techniques. First, the topological conclusions drawn in \cref{t:intro:finite-fiber} are derived from an argument based on classical ramification theory (see in particular \cref{t:residue-field-extension} which is used to prove finiteness of the residue field extension on the level of arc spaces for quasi-finite morphisms). Second, to establish the separability in positive characteristic we rely on a technical result on the cotangent map of $f_\infty$ at $\a$ (see \cref{t:differentials-cotangent-map}) which uses the structure of the sheaf of differentials on $X_\infty$ and generalizes previous results from \cites{EM09,dFD20,CdFD}.
Finally, in order to prove reducedness of arc fibers in part \eqref{i2:intro:finite-fiber} of \cref{t:intro:finite-fiber}, a separate argument involving deformation theory is needed (see \cref{t:unramified-reduced-arc-fiber}).

In general $f_\infty$ can fail to be locally of finite type
at points where the arc fibers are of finite type. 
This already occurs in the simplest possible example 
where $f \colon \A^1_k \to \A^1_k$ is the double cover given by $y=x^2$
(see \cref{ex:not-quasi-finite}).
Such pathology of $f_\infty$ is related to another interesting phenomenon
occurring in this example,
namely, the fact that the sheaf of relative differentials $\Om_{X_\infty/Y_\infty}$
has trivial fibers but nontrivial stalks at arcs stemming from
the ramification locus of $f$ (see \cref{ex:not-formally-unramified}).

The next result describes the ramification locus of $f_\infty$
and tells us when and where exactly $f_\infty$ fails to be locally of finite type.
The theorem also shows that the conclusions drawn in \cref{t:intro:finite-fiber}
are optimal.

\begin{introtheorem}[\cref{t:locally-finite-type}]
\label{t:intro:locally-finite-type}
Let $f \colon X \to Y$ be a morphism of finite type between schemes
over a perfect field $k$, and let $f_\infty \colon X_\infty \to Y_\infty$ be the induced morphism of arc spaces. 
For any $\a \in X_\infty$, the following are equivalent:
\begin{enumerate}
\item
$f_\infty$ is unramified at $\a$;
\item
$f_\infty$ is quasi-finite at $\a$;
\item
$f_\infty$ is locally of finite type at $\a$;
\item
$f$ is unramified at $\a(0)$.
\end{enumerate}
Moreover, the fiber of $f_\infty$ through $\a$ is locally of finite type at $\a$ if and only if $f$
is unramified at $\a(\e)$.
\end{introtheorem}

This theorem implies in particular that $f$ is unramified if and only if all fibers of $f_\infty$
are of finite type (see \cref{t:unramified-eq-finite-type}).

The results on arc fibers discussed above provide new tools 
to study arc spaces using projections to infinite dimensional affine spaces.
Looking at such projections is not new: for instance, it is 
the approach followed in \cite{Dri02}, where the Weierstrass preparation theorem 
is applied to eliminate variables after passing to completion;
for later applications of this approach, see also, e.g., 
\cite{BNS16,BS17,Ngo17,Reg18,MR18,Bou20}. 
What is new here is that we have gained more control on these projections
before restricting to formal neighborhoods. 

Several applications about the structure of arc spaces
are collected in \cref{s:semicontinuity-embdim,s:stable-points,s:semicont,s:div-val}. 
We review some of them here.

\subsection{Local rings at stable points}

Our first application concerns stable points of arc spaces, which are defined when $X$ is a variety. 
The notion of stable point traces back to \cite{DL99}, and their properties 
have been studied in \cites{ELM04,Reg06,dFEI08,Reg09,MR18,dFD20,Reg21,BMCS22}.
According to one of the many equivalent definitions, 
a point $\a \in X_\infty$ is \emph{stable} if it is the 
generic point of a constructible set of $X_\infty$
and is not contained in the arc space of the singular locus of $X$. 
Stable points are related to divisorial valuations on the variety, 
and among them an important class consists of the 
\emph{maximal divisorial arcs}, which 
provide the link between the geometry of $X_\infty$ and the birational geometry of $X$. 

A crucial result about stable points is the curve selection lemma
established in \cites{Reg06,Reg09}, which
provides the main tool to study the Nash problem \cite{Nas95} --
see for instance \cites{LJR12,FdBPP12,dFD16,BLM22} where the curve selection lemma is
used to solve several cases of the Nash problem. 
Reguera's proof of the curve selection lemma relies on showing
that the completed local ring $\^{\O_{X_\infty,\a}}$ at a stable point 
is Noetherian. This, in turn, was proved by showing that the maximal ideal of the reduced local ring 
$\O_{(X_\infty)_\red,\a}$ is finitely generated, and 
the question whether $\O_{(X_\infty)_\red,\a}$ may actually be Noetherian was raised in \cite{Reg09}.
Note, by contrast, that the local ring $\O_{X_\infty,\a}$ (before reduction and completion)
is not Noetherian in general; explicit examples are computed in \cref{eg:node,eg:cusp}, c.f.\ \cite[Example 3.16]{Reg09}.

Here we establish new finiteness properties of local rings at a stable points. 
In the following statement, we denote by $\dim(A)$ the Krull dimension
of a local ring $A$ and by $\embdim(A)$ its embedding dimension. 

\begin{introtheorem}[\cref{t:fg-top-noetherian}]
\label{t:intro:fg-top-noetherian}
Let $X$ be a variety over a perfect field $k$ and $\a \in X_\infty$ a stable point.
\begin{enumerate}
\item
\label{i1:intro:fg-top-noetherian}
The maximal ideal of the local ring $\cO_{X_\infty,\a}$ is finitely generated.
\item
\label{i2:intro:fg-top-noetherian}
The scheme $\Spec \cO_{X_\infty,\a}$ is topologically Noetherian and 
\[
\dim (\cO_{X_\infty,\a}) \leq \embdim (\cO_{X_\infty,\a}).
\]
\end{enumerate}
\end{introtheorem}

If $\a = \a_v$ is the maximal divisorial arc associated to a divisorial valuation $v$ on $X$, 
then the embedding dimension of the local ring is computed by
\[
\embdim(\cO_{X_\infty,\a_v}) = \^a_v(X)
\]
where the right-hand-side is the \emph{Mather log discrepancy} of $v$
defined in \cite{dFEI08}.
This formula follows from the results of \cites{Reg18,MR18} in characteristic zero, 
and was proved in full generality in \cite{dFD20}.

There is no general formula computing the Krull dimension of $\cO_{X_\infty,\a_v}$
and the only exact formula we are aware of is for the
dimension of the completion $\widehat{\cO_{X_\infty,\a_v}}$ in the case
 $X$ is a toric variety \cites{Reg21,BMCS22}. 
In general, it is a theorem of \cite{MR18} that 
if the ground field has characteristic zero then
\[
\dim (\^{\cO_{X_\infty,\a_v}}) \geq a_{v}^{\MJ}(X)
\]
where the right-hand-side is the \emph{Mather--Jacobian log discrepancy} of $v$, 
an invariant of singularities studied in \cites{Ish13,dFD14,EI15}. 
A different proof of this formula was later given in \cite{CdFD}.

Here we remove the characteristic zero assumption in this last result of Mourtada and Reguera
and extend the above formulas to all stable points.

\begin{introtheorem}[\cref{t:stable=divisorial+bounds}]
\label{t:intro:stable=divisorial+bounds}
Let $X$ be a variety over a perfect field $k$ and $\a \in X_\infty$ a stable point. 
Assume that $\a(0)$ is not the generic point of $X$. 
\begin{enumerate}
\item
The valuation $v = v_\a$ defined by $\a$ is divisorial.
\item
If $\a_{v} \in X_\infty$ is the maximal divisorial arc associated to $v$
and $c = \codim(\a,\a_v)$, then
\[
\embdim(\O_{X_\infty,\a}) = c + \embdim(\O_{X_\infty,\a_v}) = c + \^a_{v}(X)
\]
and
\[
\dim (\^{\cO_{X_\infty,\a}}) \geq c + a_{v}^{\MJ}(X).
\]
Moreover, $\codim(\a,\a_v)$ is birationally invariant, in the sense that
for every proper birational morphism $X' \to X$, if
$\a'$ and $\a_v'$ are the lifts of $\a$ and $\a_v$ to $X'_\infty$ then
$\codim(\a',\a_v') = \codim(\a,\a_v)$. 
\end{enumerate}
\end{introtheorem}

A key ingredient in the proof of \cref{t:intro:stable=divisorial+bounds} is the semicontinuity
property of embedding dimension, which is discussed next.

\subsection{Semicontinuity properties}

In our recent papers \cites{dFD20,CdFD}, we proposed to look at two 
invariants of local rings as tools to study the local structure of arc spaces:
the \emph{embedding dimension} $\embdim(\O_{X_\infty,\a})$ of the local ring
at a point $\a \in X_\infty$ (which we already encountered in this introduction), 
and its \emph{embedding codimension} $\embcodim(\O_{X_\infty,\a})$. 

It was proved in \cite{dFD20} that
the embedding dimension of $\O_{X_\infty,\a}$ measures the \emph{jet codimension} of $\a$, and 
its finiteness characterizes {stable points}. 
Since in general the local ring $\O_{X_\infty,\a}$ is not Noetherian, 
the notion of embedding codimension is more subtle in this context compared to the 
more familiar Noetherian setting. It was recently studied in \cite{CdFD}, 
where it is shown that its finiteness characterizes arcs that are not entirely 
contained in the singular locus of $X$. 

These invariants were studied for Noetherian local rings in \cite{Lec64}, where
it was proved that they satisfy a certain semicontinuity property. 
To the best of our knowledge, the 
question whether Lech's result extends beyond the Noetherian setting remains open.

The next result extends Lech's semicontinuity theorem to local rings of arc spaces
of schemes of finite type, a setting that provides
many interesting examples of non-Noetherian rings. 
This is particularly relevant if one wants to consider these invariants
as measures of singularities of arc spaces,
a point of view that was adopted in \cite{CdFD}.

\begin{introtheorem}[\cref{t:semicont}]
\label{t:intro:semicont}
Let $X$ be a scheme locally of finite type over a perfect field $k$. 
Let $\a,\a' \in X_\infty$ be two points with $\a'$ specializing to $\a$, 
and let $c = \codim(\a,\a')$.
\begin{enumerate}
\item
\label{i1:intro:semicont}
We have
\[
\embdim(\O_{X_\infty,\a}) \ge c + \embdim(\O_{X_\infty,\a'}).
\]
\item
\label{i2:intro:semicont}
If $\embdim(\O_{X_\infty,\a}) < \infty$
and $\dim(\gr(\O_{X_\infty,\a})) \le c + \dim(\gr(\O_{X_\infty,\a'}))$, 
then 
\[
\embcodim(\O_{X_\infty,\a}) \ge \embcodim(\O_{X_\infty,\a'}).
\]
\end{enumerate}
\end{introtheorem}

We recover from \cref{t:intro:semicont} the fact that the Mather log discrepancy
is upper-semicontinuous with respect to the order 
among divisorial valuations on a variety $X$ determined by specialization in $X_\infty$
of the associated maximal divisorial arcs (see \cref{t:semicont-^A-arc-top}), 
a property that can also be deduced from the results of \cite{dFEI08,dFD20}.
Such order among divisorial valuations is at the heart of the Nash problem \cite{Nas95,FdBPP12}
and its generalized formulation \cite{ELM04,Ish08,FdBPPPP17,BdlBdLFdBP22}, 
and for this reason it has been referred to as the \emph{Nash order} \cite{BLM22}; 
it is furthermore implicit in the applications of arc spaces to 
problems of lower-semicontinuity for minimal log discrepancies \cite{EMY03,EM04}
and minimal Mather--Jacobian log discrepancies 
\cite{Ish13,dFD14}. The fact that Mather log discrepancies are upper-semicontinuous
with respect to this order is in line with Ishii's lower-bound on 
minimal Mather log discrepancies proved in \cite{Ish13}. 

It is unclear what should replace the condition put in part \eqref{i2:intro:semicont}
of \cref{t:intro:semicont} if $\embdim(\O_{X_\infty,\a}) = \infty$ (i.e., when $\a$ is not stable), 
and it is natural to wonder about the behavior of the embedding codimension in this range.
As a partial result in this direction, we prove that the embedding codimension defines a function 
on the $k$-rational points of the open set $X_\infty^\o := X_\infty \setminus (\Sing X)_\infty$ 
that is locally constructible and is lower-semicontinuous 
on each stratum of a natural stratification of this space (see \cref{t:embcodim-constr} for details).
The proof uses results from \cite{GK00,Dri02,CdFD}.

\subsection{Closing remarks}
The arc space $X_\infty$ of a scheme $X$ captures, in a way 
that is not yet fully understood, certain features of the singularities of $X$.
Over the years, the geometry of arc spaces
has been investigated from many different points of view, especially when $X$ is a variety, and several discoveries 
have shown that $X_\infty$, while in general
an infinite dimensional object that is far from being Noetherian or 
of finite type, still manifests subtle finiteness properties. Notable examples are:
Greenberg's theorem implying the constructibility of the sets of liftable jets \cite{Gre66};
the theorems of Nash and Kolchin on the topology of $X_\infty$ and the families
of arcs through the singularities of $X$ \cite{Nas95,Kol73};
Drinfeld--Grinberg--Kazhdan's theorem 
on the structure of the formal neighborhoods of $X_\infty$ at its nondegenerate rational points \cite{GK00,Dri02};
and the aforementioned result of Reguera on the formal neighborhoods of
$X_\infty$ at its stable points \cite{Reg06,Reg09}.
The results of this paper provide new manifestations of finiteness in arc spaces.

\subsection*{Acknowledgements} 

We thank David Bourqui, Devlin Mallory and Johannes Nicaise for useful comments. We also want to thank the referee for many valuable comments and especially for bringing to our attention \cite[Remark 3.5]{MR18}, which allowed us to strengthen \cref{t:intro:stable=divisorial+bounds}.

\section{Preliminaries on arc spaces}

Let $X$ be a scheme over a field $k$. The \emph{arc space} $X_\infty$
of $X$ is the scheme over $k$ representing the functor of points given, for any
$k$-algebra $R$, by $R \mapsto \liminv_m X_m(R)$ where
$X_m(R)=\Hom_k(\Spec R[t]/(t^{m+1}),X)$ is the functor of points of the $m$-th
jet scheme of $X$ over $k$. By \cite[Remark~4.6]{Bh16}, the functor $X_\infty(R)$ is
naturally isomorphic to $\Hom_k(\Spec R[[t]],X)$. 
If $X = \Spec A$ where $A$ is a $k$-algebra, then $X_\infty = \Spec A_\infty$
where $A_\infty := \HS^\infty_{A/k}$ is the algebra of 
(Hasse--Schmidt) higher
differentials of $A$ over $k$. 
For a comprehensive introduction to the subject we refer the reader to \cite{Voj07,EM09a}. 

When we want to specify the ground field in the notation, 
we write $(X/k)_\infty$ for the arc space of $X$ over $k$. 
Given a point $x \in X$ on a scheme over a field $k$, if the ground field is
clear from the context then 
we will denote by $x_\infty$ the arc space of $\Spec k_x$ over $k$. 
For a field extension $L/K$ we denote by $(L/K)_\infty$ 
the arc space of $\Spec L$ regarded as a scheme over $K$. 

A point $\a \in X_\infty$ is called an \emph{arc} on $X$ and corresponds to a
morphism $\Spec k_\a[[t]] \to X$ where $k_\a$ is the residue field of $\a$. 
Conversely, any $k$-morphism $\a \colon \Spec K[[t]] \to X$, where $K/k$ is a field extension, 
will be regarded as a $K$-valued point of $X_\infty$ and will be called
a \emph{$K$-valued arc} (or simply an \emph{arc}) on $X$.

Given an arc $\a \colon \Spec K[[t]] \to X$, we consider the morphisms
$\a_0 \colon \Spec K \to X$ and $\a_\e \colon \Spec \LS{K} \to X$
induced by restriction, and let $\a(0), \a(\e) \in X$ be their images.
We call $\a(0)$ the \emph{special point} of $\a$ and $\a(\e)$ the \emph{generic point} of $\a$. 

Any arc $\a \colon \Spec K[[t]] \to X$ defines a 
semi-valuation $\ord_\a$ on the local ring $\O_{X,\a(0)}$, and hence
on $\O_X(U)$ for any open set $U \subset X$ containing $\a(0)$, 
by setting $\ord_\a(g) := \ord_t(\a^\sharp(g))$. 
Note that if $X$ is a variety and $\a(\e)$ is the generic point of $X$, then 
$\ord_\a$ extends to a $\Z$-valued valuation of the function field of $X$
which will be denoted by $v_\a$. 
We say that an arc $\a \in X_\infty$ has \emph{order of contact $q$} with a closed 
subscheme $Z \subset X$ if $\ord_\a(\I_Z) = q$ where $\I_Z \subset \O_X$ is the ideal sheaf of $Z$; 
if the set of such arcs is irreducible, then we we refer to its generic point as the 
\emph{generic arc with order of contact $q$} with $Z$.

%
%

\section{Finiteness of residue field extensions}

\label{s:residue-field-ext-finite}

The purpose of this section is to establish a finiteness result at the level of 
residue fields in \cref{t:residue-field-extension}. Loosely speaking, it states that given a morphism of schemes $f \colon X \to Y$
and an arc $\a \in X_\infty$, certain finiteness properties of the residue field extension of 
the generic point of the arc $\a(\e)$ over its image in $Y$ lift to analogous properties
of the residue field extension of $\a$ over its image in $Y_\infty$. 
The theorem lays the foundations for results about finiteness of arc fibers that are
proved in the following sections.

\begin{theorem}
\label{t:residue-field-extension}
Let $f \colon X \to Y$ be a morphism of schemes over a perfect field $k$. Let $\a \in X_\infty$ and $\b = f_\infty(\a)$, and write $x = \a(\eta)$, $y = \b(\eta)$. 
\begin{enumerate}
\item 
\label{i1:residue-field-extension}
If $k_x/k_y$ is algebraic, then so is $k_\a/k_\b$.
\item 
\label{i2:residue-field-extension}
If $k_x/k_y$ is finite separable, then so is $k_\a/k_\b$.
\end{enumerate}
\end{theorem}

We note that assertion \eqref{i1:residue-field-extension} holds over any field (see \cref{l:same-alg-closure});
the assumption that $k$ be perfect is needed to prove \eqref{i2:residue-field-extension}. 
The following example shows that in \eqref{i2:residue-field-extension} 
of \cref{t:residue-field-extension} the extension $k_\a/k_\b$ can fail to be finite 
if $k_x/k_y$ is finite but not separable.

\begin{example}
\label{ex:counterex-res-field-finite}
Let $k$ be a field of characteristic $p > 0$. Let $X = Y = \A^1_k = \Spec k[x]$ and $f \colon X \to Y$
be the $k$-morphism given by $x \mapsto x^p$.
Consider the generic arc $\a$ of $X$, defined by $x \mapsto \sum_{n\geq 0} x_n t^{n}$ where $x_n \in k_\a$
is the $n$-th higher derivative of $x$. 
Its image $\b = f_\infty(\a)$ is defined by
$x \mapsto \sum_{n\geq 0} x_n^p t^{np}$.
Thus the extension of residue fields $k_\a/k_\b$ is given by
$k(x_n^p \mid n\in \N) \subset k(x_n \mid n \in \N)$, 
which is not finitely generated.
\end{example}

The proof of \cref{t:residue-field-extension}, which will be finished in the next section, proceeds as follows: we will first prove assertion \eqref{i1:residue-field-extension} in \cref{l:same-alg-closure}; we will then prove the finiteness part in assertion \eqref{i2:residue-field-extension} in \cref{p:separable-field-extension}; finally, the separability part in \eqref{i2:residue-field-extension} will follow from \cref{t:differentials-cotangent-map}, whose proof depends on \cref{p:separable-field-extension} but also requires results from \cite{dFD20} and in particular the base field $k$ to be perfect.
We start with a lemma. 

\begin{lemma}
\label{l:alg-ext-laurent}
Let $K/L$ be a field extension. Let $\widetilde L$ be the relative algebraic closure of $L$ in $K$, 
and similarly write $\widetilde {\LS{L}}$ for the relative algebraic closure of $\LS{L}$ inside $\LS{K}$. Then
\begin{equation}
\label{eq:alg-ext-laurent}
\bigcup_{L'} \LS{L'}
\subset 
\widetilde {\LS{L}}
\subset
\LS{\widetilde L}
\end{equation}
where the union is taken over all intermediate field extensions $L \subset L' \subset K$ that are finite over $L$.
In particular, if $K/L$ is finitely generated then $\widetilde {\LS{L}} = \LS{\widetilde L}$.
\end{lemma}

\begin{proof}
Regarding the first inclusion in \eqref{eq:alg-ext-laurent}, 
it suffices to observe that for any extension $L(\xi)/L$ where
$\xi \in K$ algebraic over $L$, we have $\LS{L(\xi)} = \LS{L}(\xi)$, which is a finite extension of $\LS{L}$.

To prove the second inclusion in \eqref{eq:alg-ext-laurent},  
let $f = a_{n_0} t^{n_0} + a_{n_1} t^{n_1} + \cdots$ be an element of $\LS{K}$
that is algebraic over $\LS{L}$, where $a_{n_i} \ne 0$ and $n_i < n_{i+1}$. We
want to show that $a_{n_i} \in \widetilde L$ for all $i$. Let $P(X) \in \LS{L}[X]$  be a
monic polynomial such that $P(f) = 0$. Consider the $t$-expansion $P(f) =
\sum_{m \geq m_0} P_m t^m$ with $P_{m_0} \ne 0$. Then $P_{m_0}$ is a polynomial
in $a_{n_0}$ with coefficients in $L$, showing that $a_{n_0} \in \widetilde L$. This
implies that $f_1 = f - a_{n_0}t^{n_0}$ is also algebraic over $\LS{L}$, and
the same argument shows that $a_{n_1} \in \widetilde L$. Recursively, we see that
$a_{n_i} \in \widetilde L$ for all $i$, as required.

Finally, if $K$ is any finitely generated extension of $L$, then so is any intermediate
field $K \supset L' \supset L$. In particular $\widetilde L/L$ is finite, and so the
equality $\widetilde {\LS{L}} = \LS{\widetilde L}$ follows from \eqref{eq:alg-ext-laurent}. 
\end{proof}

\begin{remark}
If $K/L$ is not finitely generated, then the second inclusion in \eqref{eq:alg-ext-laurent}
is strict in general. 
To see this, let $L = \Q$ and consider an infinite chain of algebraic extensions $K_n \subset K_{n+1}$ where
$K_0 = L$ and 
for each $n$ there exists $x_n \in K_n$ whose minimal polynomial over $K_{n-1}$ has degree $n$. 
Let then $K = \bigcup K_n$. We consider the series $f = \sum_{n\geq 0} x_n t^n \in \LS{\widetilde L}$. 
Assume that $f$ is algebraic over $\LS{L}$. 
Let $P(X) \in \LS{L}[X]$ be the minimal polynomial of $f$ and let $d$ be its degree. 
Taking the $t$-adic expansion of $P(f)$ we get
\[
P(f) = \sum_{i \geq 0} P_i(x_1,\ldots,x_i) t^i = 0,
\]
where $P_i \in L[X_1,\ldots,X_i]$ is of degree $d$ again. 
Choose $i$ such that $P_i(X_1,\ldots,X_i)$ has a nonzero term involving some $X_m$ with $m > d$. Let $n$ be the largest such $m$. Then $P_i(x_1,\ldots,x_i) = 0$ gives a contradiction to the assumption that the minimal polynomial of $x_n$ over $K_{n -1}$ has degree $n > d$. 

\end{remark}

\begin{remark}
The first inclusion in \eqref{eq:alg-ext-laurent} is also strict in general, at least
when $L$ is a field of positive characteristic.
For an example, let $L = k(x_i^p \mid i \in \N) \subset K = k(x_i \mid i \in I)$, where $k$ is a field of characteristic $p > 0$, and consider the element $f = \sum_{i \geq 0 } x_i t^i$, whose coefficients generate $K$ and thus do not belong to a subfield that is finite over $L$. We do not know 
any characteristic zero example
where the first inclusion in \eqref{eq:alg-ext-laurent} is strict.
\end{remark}

\begin{proposition}    
\label{l:same-alg-closure}
Let $f \colon X \to Y$ be a morphism of schemes over a field $k$. 
Let $\a \in X_\infty$ and $\b = f_\infty(\a)$, and write $x = \a(\eta)$, $y = \b(\eta)$. 
If $k_x/k_y$ is algebraic, then so is $k_\a/k_\b$, and $\a$ is a closed point of the fiber $f_\infty^{-1}(\b)$.
\end{proposition}

\begin{proof}
With the same notation as in \cref{l:alg-ext-laurent}, let $\widetilde {k_\b}$ be the relative algebraic closure of $k_\b$ in $k_\a$, so that $\LS{\widetilde {k_\b}}$ contains the relative algebraic closure of $\LS{\k_\b}$ in $\LS{\k_\a}$. We have the solid arrows in the following diagram:
\[\xymatrix@R=20pt{
k_x \ar^-{\a_\e^\sharp}[r]
\ar@{-->}^{\vp}[rd]
& *+[r]{\LS{k_\a}}
\\
& *+[r]{\LS{\widetilde {k_\b}}} \ar[u]
\\
k_y \ar^-{\b_\e^\sharp}[r] \ar^{f^\sharp}[uu]
& *+[r]{\LS{\k_\b}} \ar[u]
}\]
By hypothesis $f^\sharp$ gives an algebraic extension, and we see that the image $\a_\e^\sharp(k_x)$ is algebraic over $\LS{\k_\b}$, hence contained in $\LS{\widetilde {k_\b}}$. In other words, the dashed arrow $\vp$
making the diagram commutative exists.
Since $\a_\e^\sharp$ is induced by an arc $\a \colon \Spec k_\a[[t]] \to X$, we see that $\vp$ is of the form $\vp=\g_\e^\sharp$ for some arc $\g \colon \Spec \widetilde {k_\b}[[t]] \to X$. At the level of arc spaces we get a factorization $\Spec k_\a \to \Spec \widetilde {k_\b} \to X_\infty$, showing that $k_\a = \widetilde {k_\b}$.

The last assertion in the proposition is a standard fact about morphisms of schemes, 
see \cite[\href{https://stacks.math.columbia.edu/tag/01TE}{Tag 01TE}]{stacks-project}.
\end{proof}

\begin{proposition}
\label{p:separable-field-extension}
Let $f \colon X \to Y$ be a morphism of schemes over a field $k$. 
Let $\a \in X_\infty$ and $\b = f_\infty(\a)$, and write $x = \a(\eta)$, $y = \b(\eta)$. 
If $k_x/k_y$ is finite separable, then $k_\a/k_\b$ is finite.
\end{proposition}

\begin{proof}
After replacing $X$ and $Y$ by the closures of $x$ and $y$, we can assume that $X$ and $Y$ are integral schemes with generic points $x$ and $y$. We identify $x$ and $y$ with the schemes $\Spec k_x$ and $\Spec k_y$, and consider their arc spaces $x_\infty$ and $y_\infty$ over $k$. Since the extension $k_x/k_y$ is finite separable, the map $x \to y$ is finite \'etale and we have that $x_\infty = y_\infty \times_y x$. In particular, the map $x_\infty \to y_\infty$ is also finite \'etale.

Consider the map $\vp_\a \colon \Spec \k_\a[[s,t]] \to X$
given by $\vp_\a(s,t) := \a(s+t)$.
Thinking of $\vp_\a$ as an arc $\Spec \k_\a[[s]] \to X_\infty$ on the arc space of $X$, 
we consider its generic point $\a' = \vp_\a(\eta) \in X_\infty$. Its residue field comes equipped with an inclusion $k_{\a'} \subset \LSS{\k_\a}{s}$. Moreover, $\vp_\a$ gives a local ring map $\O_{X_\infty,\a} \to k_\a[[s]]$ which induces an isomorphism on residue fields, and whose image is contained in $k_{\a'}$. Therefore the ring $\O_v = k_{\a'} \cap k_\a[[s]]$ is a valuation ring of $k_{\a'}$ with residue field $\O_v/\fm_v = k_\a$, with the valuation $v$ being the restriction of $\ord_s$ to $k_{\a'}$. Notice that $\a'(0) = \a'(\eta) = x$, so in fact $\a' \in x_\infty$.

We construct in a similar way a point $\b' \in y_\infty$ such that $f_\infty(\a') = \b'$, and we get an extension $k_{\a'}/k_{\b'}$ compatible with the extension $\LSS{\k_\a}{s}/\LSS{\k_\b}{s}$. Moreover, $k_{\b'}$ comes equipped with a valuation ring $\O_w = k_{\b'} \cap k_\b[[s]]$ with residue field $\O_w/\fm_w = k_\b$, and $\O_w$ is the restriction of $\O_v$ to $k_{\b'}$.

Since $\a' \in x_\infty$, $\b' \in y_\infty$, $f_\infty(\a') = \b'$, and the map $x_\infty \to y_\infty$ is \'etale, we see that the extension $k_{\a'}/k_{\b'}$ is finite separable by \cite[\href{https://stacks.math.columbia.edu/tag/02GL}{Tag 02GL}]{stacks-project}.

To conclude, notice that $k_{\a'}/k_{\b'}$ is finite and that $\O_v$ is an extension of $\O_w$, so the finiteness of the extension of residue fields $k_\a/k_\b$ follows from \cite[\href{https://stacks.math.columbia.edu/tag/0ASH}{Tag 0ASH}]{stacks-project} or \cite[\protect{Chapter~VI, Section~6, Corollary~2}]{ZS75}.
\end{proof}

%
%

\section{Cotangent maps at the level of arc spaces}

\label{s:cotangent-map}

The main result of this section aims to describe the cotangent map of
$f_\infty$ at $\a$ when
$f$ is either smooth or unramified at $\a(\e)$. 
We will use this result to complete the proof of \cref{t:residue-field-extension}.
The result generalizes \cite[Theorem~1.2]{EM09}, \cite[Theorem 9.2]{dFD20} and \cite[Theorem 8.1]{CdFD}.
A specific extension of the latter is given in \cref{c:cotangent-map-isom}, which will be used to provide an analogous extension of \cite[Theorem 8.5]{CdFD} in \cref{t:embcodim-bound}.

Before stating the main result, we recall the definition of unramified morphism, 
in which we follow \cite[\href{https://stacks.math.columbia.edu/tag/02G4}{Tag 02G4}]{stacks-project}. 

\begin{definition}
\label{d:unramified}
A morphism of schemes $f \colon X \to Y$ is \emph{unramified} at a point $x \in X$ 
if $f$ is locally of finite type at $x$ and $\Om_{X/Y,x} = 0$. 
The morphism is \emph{unramified} if it is unramified at every point. 
\end{definition}

\begin{remark}
The above definition is weaker than the one given in \cite[(17.3.1)]{ega-iv-pt4}, 
where an unramified morphism is assumed to be locally of finite presentation.
In particular, any closed immersion is unramified in our sense.
\end{remark}

\begin{theorem}
\label{t:differentials-cotangent-map}
Let $f \colon X \to Y$ be a morphism of schemes over a perfect field $k$. Let $\a \in X_\infty$ and $\b = f_\infty(\a)$, and consider the cotangent map
\[
T^*_\a f_\infty \colon \fm_\b/\fm_\b^2 \otimes_{k_\b} k_\a \to \fm_\a/\fm_\a^2.
\]
\begin{enumerate}
\item 
\label{i2:differentials-cotangent-map}
If $f$ is unramified at $\a(\eta)$, then the residue field extension $k_\a/k_\b$ is finite separable and $T^*_\a f_\infty$ is surjective.
\item 
\label{i1:differentials-cotangent-map}
If $X$ and $Y$ are locally of finite type over $k$ and $f$ is smooth at $\a(\eta)$, then
\[
\dim (\ker T^*_\a f_\infty) \leq \ord_\a (\Fitt^r \Om_{X/Y}),
\]
where $r = \dim(\Om_{X/Y} \otimes k_{\a(\e)})$, and equality holds if $X$ is smooth at $\a(0)$
and $f$ is \'etale at $\a(\e)$. 
\end{enumerate}
\end{theorem}

We now have all the pieces in place to complete the proof of the main theorem from the previous section. 

\begin{proof}[Proof of \cref{t:residue-field-extension}]
As we already mentioned, part \eqref{i1:residue-field-extension} follows by \cref{l:same-alg-closure}
and the finiteness of $\k_\a/\k_\b$ stated in \eqref{i2:residue-field-extension} is proved in
\cref{p:separable-field-extension}. As for the separability of $\k_\a/\k_\b$, 
we reduce as in the proof of \cref{p:separable-field-extension} to the 
case where $X$ is an integral scheme and $x$ is its generic point. 
Then by \cite[\href{https://stacks.math.columbia.edu/tag/090W}{Tag 090W}]{stacks-project} the assumption that $\k_x/\k_y$
is finite separable implies that $f$ is unramified at $x$,
hence separability follows from \eqref{i2:differentials-cotangent-map} of \cref{t:differentials-cotangent-map}.
\end{proof}

As we mentioned in \cref{s:residue-field-ext-finite}, part \eqref{i2:differentials-cotangent-map} of \cref{t:differentials-cotangent-map} requires \cref{p:separable-field-extension} for the finiteness of the extension $k_\a/k_\b$. The other main ingredient is the description of the sheaf of differential of the arc space established in \cite{dFD20}. Let us review the results there needed for the proof of \cref{t:differentials-cotangent-map}. For $\a \in X_\infty$ we are interested in the fiber $\Om_{X_\infty/k} \otimes k_\a$.
Clearly we may assume $X = \Spec A$ is affine. Set $B_\a = k_\a[[t]]$ and $P_\a = \LS{k_\a}/t \cdot k_\a[[t]]$. Then by \cite[Remark 5.4]{dFD20} we have
\[
\Om_{X_\infty/k} \otimes k_\a \isom (\Om_{X/k} \otimes B_\a) \otimes_{B_\a} P_\a,
\]
where $B_\a$ is an $A$-module via the universal arc $X_\infty \^\times \Spf k[[t]] \to X$ given by the ring map $A \to A_\infty[[t]]$. As an intermediate step we will compute $\Om_{X/k} \otimes B_\a$ first. 
A key ingredient in the computation is the structure theorem for modules over principal ideal domains, 
which implies that for any finitely generated $A$-module $M$ we have
$M \otimes B_\a \isom F_\a \oplus T_\a$
where $F_\a$ is free of rank $d = \dim_{k_{\a(\e)}} (M \otimes k_{\a(\e)})$
and $T_\a$ is torsion, of dimension over $k_\a$ equal to $\ord_\a (\Fitt^d M)$ 
where $\Fitt^d M$ denotes the $d$-th Fitting ideal of $M$. 
We refer the reader to \cite[Section~6]{dFD20} for more details.

\begin{proof}[Proof of \cref{t:differentials-cotangent-map}]
We may assume that $X$ and $Y$ are affine. Consider the following diagram over $k_\a$ where rows and columns are exact:
\begin{equation}
\label{eq:diff-diagram}
\xymatrix{ 
          0 \ar[r] & \fm_\b/\fm_\b^2 \otimes k_\a \ar[r] \ar[d]^-{T^*_\a f_\infty} & \Om_{Y_\infty/k} \otimes k_\a \ar[r] \ar[d]^-{\Phi_\a} & \Om_{k_\b/k} \otimes k_\a \ar[r] \ar[d]^-{\Theta_\a} & 0\\
          0 \ar[r] & \fm_\a/\fm_\a^2 \ar[r] & \Om_{X_\infty/k} \otimes k_\a \ar[r] \ar[d] & \Om_{k_\a/k} \ar[r] \ar[d] & 0\\
          & & \Om_{X_\infty/Y_\infty} \otimes k_\a \ar[r] \ar[d] & \Om_{k_\a/k_\b} \ar[d] \ar[r] & 0\\
          & & 0 & 0 &}
\end{equation}

Note that exactness of the middle row follows from the assumption that $k$ is perfect.
By naturality of the formula in \cite[Theorem 5.3]{dFD20} we can compute the map $\Phi_\a$ from the exact sequence
\begin{equation}
\label{eq:first-fundamental}
\xymatrix@C=15pt{ \Om_{Y/k} \otimes \cO_X \ar[r]^-{\f} & \Om_{X/k} \ar[r] & \Om_{X/Y} \ar[r] & 0}       
\end{equation}
by tensoring first with $B_\a = k_\a[[t]]$ and then with $P_\a = \LS{k_\a}/t \cdot k_\a[[t]]$.
Tensoring with $B_\a$ yields
\begin{equation}
\label{eq:diff-diagram-B}
\xymatrix@C=15pt{\Om_{Y/k} \otimes B_\a \ar[r]^-{\f_\a} & \Om_{X/k} \otimes B_\a \ar[r] & \Om_{X/Y} \otimes B_\a  \ar[r] & 0,}
\end{equation}
and tensoring with $P_\a$ yields
\begin{equation}
\label{eq:diff-diagram-P}
\xymatrix@C=15pt{\Om_{Y/k} \otimes P_\a \ar[r]^-{\Phi_\a} & \Om_{X/k} \otimes P_\a \ar[r] & \Om_{X/Y} \otimes P_\a  \ar[r] & 0,}
\end{equation}
which agrees with the middle column in the diagram \eqref{eq:diff-diagram}.

To prove assertion \eqref{i2:differentials-cotangent-map}, assume that $f$ is unramified at $\a(\e)$. We claim that it is sufficient to prove that the map $\Phi_\a$ in \eqref{eq:diff-diagram} is surjective. Indeed, observe that then $\Theta_\a$ is surjective as well and hence $\Om_{k_\a/k_\b} = 0$. Since the field extension $k_{\a(\e)}/k_{\b(\e)}$ is finite separable, it follows by \cref{p:separable-field-extension} that the extension $k_\a/k_\b$ is finite. Using \cite[\href{https://stacks.math.columbia.edu/tag/090W}{Tag 090W}]{stacks-project} we see that
$k_\a/\k_\b$ is separable, hence $\Theta_\a$
is injective by \cite[Theorem~26.6]{Mat89}.
We conclude by the snake lemma that $T^*_\a f_\infty$ is surjective.

In order to show that $\Phi_\a$ is surjective, first note that 
the assumption that $f$ is unramified at $\a(\e)$ implies that
$\Om_{X/Y} \otimes B_\a$ is torsion, hence $\Om_{X/Y} \otimes P_\a = 0$. 
Then the surjectivity of $\Phi_\a$ follows by \eqref{eq:diff-diagram-P}.

We now address \eqref{i1:differentials-cotangent-map}, and
henceforth assume that $X$ and $Y$ are of finite type and 
$f$ is smooth at $\a(\eta)$. By \eqref{eq:diff-diagram}, 
$\dim \ker T_\a^*f_\infty \le \dim \ker \Phi_\a$, hence it suffices to prove that 
$\dim \ker \Phi_\a\leq \ord_\a (\Fitt^e \Om_{X/Y})$. 
The key point it to show that even though \eqref{eq:diff-diagram-B} in general does not extend
to an short exact sequence to the left
(since we are not assuming $f$ to be smooth at $\a(0)$), we can still extend \eqref{eq:diff-diagram-P}
on the left to an exact sequence 
\begin{equation}
\label{eq:diff-diagram-P-Tor1}
\xymatrix@C=15pt{\Tor_1^A(\Om_{X/Y},P_\a) \ar[r] 
& \Om_{Y/k} \otimes P_\a \ar[r]^-{\Phi_\a} & \Om_{X/k} \otimes P_\a \ar[r] & \Om_{X/Y} \otimes P_\a  \ar[r] & 0,}
\end{equation}
which will assist us in controlling the kernel of $\Phi_\a$.

To see this, write
$\Om_{Y/k} \otimes B_\a = F_Y \oplus T_Y$,  $\Om_{X/k} \otimes B_\a = F_X \oplus T_X$ and $\Om_{X/Y} \otimes B_\a = F_{X/Y} \oplus T_{X/Y}$, where $F_Y$, $F_X$, $F_{X/Y}$ are free and $T_Y$, $T_X$, $T_{X/Y}$ are torsion modules. 
Since $f$ is smooth at $\a(\e)$, pulling back \eqref{eq:first-fundamental} 
along the generic point of the arc $\a$ yields a short exact sequence -- that is, the first map
becomes injective. This implies that the restriction of $\f_\a$ to the free part $F_Y$ is injective, hence
$\ker\f_\a$ is contained in the torsion part $T_Y$. 
Setting for short $\ov T_Y = T_Y/\ker\f_\a$, we obtain a short exact sequence
\[
\xymatrix@C=15pt{0 \ar[r] & F_Y \oplus \ov T_Y \ar[r]^-{\ov{\f}_\a} & F_X \oplus T_X \ar[r] & F_{X/Y} \oplus T_{X/Y} \ar[r] & 0.}
\]
Note that by definition of $P_\a$ we have an exact sequence
\begin{equation}
\label{eq:definition-twisting}
\xymatrix@C=15pt{0 \ar[r] & B_\a \ar[r]^-{\cdot t} & \LS{k_\a} \ar[r] & P_\a \ar[r] & 0. }
\end{equation}
Thus $P_\a$ is a divisible $B_\a$-module and tensoring with $P_\a$ kills torsion. 
Therefore we get a long exact sequence
\begin{equation}
\label{eq:long-exact-seq}
\xymatrix@C=15pt{\Tor^{B_\a}_1(T_{X/Y},P_\a) \ar[r] 
&F_Y \otimes P_\a \ar[r] & F_X \otimes P_\a \ar[r] & F_{X/Y} \otimes P_\a \ar[r] & 0.}
\end{equation}
This sequence agrees with \eqref{eq:diff-diagram-P-Tor1}, and shows that
$\dim \ker\Phi_\a \leq \dim \Tor^{B_\a}_1(T_{X/Y},P_\a)$. From the sequence \eqref{eq:definition-twisting} we see that $\Tor^{B_\a}_1(T_{X/Y},P_\a) \isom T_{X/Y}$, which is of dimension $\ord_\a (\Fitt^r \Om_{X/Y})$. 
This proves the first assertion of \eqref{i1:differentials-cotangent-map}. 

For the second assertion of \eqref{i1:differentials-cotangent-map}, 
first note that as we are now assuming that $f$ is \'etale at $\a(\e)$ 
the map $\Theta_\a$ in \eqref{eq:diff-diagram} is injective, hence
$\dim \ker T_\a^*f_\infty = \dim \ker \Phi_\a$. To conclude, note that
if $X$ is smooth at $\a(0)$ then $T_X = 0$. This implies that the first map in \eqref{eq:long-exact-seq} is injective, 
hence we have $\dim \ker\Phi_\a = \dim \Tor^{B_\a}_1(T_{X/Y},P_\a)$ in this case. 
\end{proof}

\begin{remark}
\label{r:differentials-cotangent-map}
In the setting of \eqref{i1:differentials-cotangent-map} of \cref{t:differentials-cotangent-map}, 
by writing more terms of the long exact sequence \eqref{eq:long-exact-seq}
one can also see that the following bound holds:
\[
\dim \ker T^*_\a f_\infty \leq \ord_\a (\Fitt^r \Om_{X/Y}) - \ord_\a (\Fitt^d \Om_{X/k}) + 
\ord_\b (\Fitt^e \Om_{Y/k})
\]
where $d = \dim (\Om_{X/k} \otimes k_{\a(\e)})$ and
$e = \dim (\Om_{Y/k} \otimes k_{\b(\e)})$.
\end{remark}

Given a finite morphism of varieties $f \colon X \to Y$,  
\cref{t:differentials-cotangent-map} says that for any $\a$ such that $f$ is \'etale at $\a(\eta)$ the cotangent map $T^*_\a f_\infty$ is surjective with finite-dimensional kernel. By contrast, the next example shows that if $f$ is only unramified at $\a(\eta)$ then the kernel of $T^*_\a f_\infty$ can be infinite-dimensional.

\begin{example}
Consider the nodal curve singularity $Y$ given by $y^2 - x^2 (x+1) = 0$ in $\A^2_k$. Let $f \colon X \to Y$ be the normalization, which is given in coordinates by
$u \mapsto (u^2 - 2u, (u - 1)(u^2 - 2u))$.
Note that $f$ is unramified. Let $\b \in Y_\infty(k)$ be the constant arc centered at the node
and $\a \in X_\infty(k)$ a constant arc on $X$ mapping to $\b$. 
To compute the cotangent map we proceed as in the proof of \cref{t:differentials-cotangent-map}. 
Adopting the same notation as in the proof of the theorem, 
the map $\Phi_{\a}$ in \eqref{eq:diff-diagram} is given by
\[
\xymatrix@C=15pt{(\Om_{Y/k} \otimes k) \otimes_k P_{\a} \ar[r] & (\Om_{X/k} \otimes k) \otimes_k P_{\a}.}
\]
Here $\Om_{Y/k} \otimes k$ is a free module of rank $2$, generated by $dx$, $dy$,
and $\Om_{X/k} \otimes k$ is free of rank $1$ generated by $du$, with
$\Phi_{\a}$
determined by $dx \mapsto 2 du$, $dy \mapsto -2du$.
In particular, the kernel of $\Phi_{\a}$ is infinite-dimensional over $k$. As the map $\Theta_{\a}$ in \eqref{eq:diff-diagram} is clearly an isomorphism
in this example, we see that $\ker T^*_{\a} f_\infty \isom \ker \Phi_{\a}$.
\end{example}

\cref{t:differentials-cotangent-map} 
allows us to remove the assumption on the characteristic in
\cite[Theorem 8.1]{CdFD}, which in the case 
of positive characteristics was restricted to $k$-rational arcs $\a \in X_\infty(k)$. 
Here we deal explicitly only with the case where $X$ is smooth at $\a(\eta)$, 
but the general case can be obtained similarly.

\begin{corollary}
\label{c:cotangent-map-isom}
Let $X \subset \A^n_k$ be an affine scheme over a perfect field $k$. 
Let $\a \in X_\infty$ and assume that $X$ is smooth at $\a(\eta)$. Then there exists 
a morphism $f \colon X \to \A^d_k$ that is \'etale at $\a(\eta)$ and such that the cotangent map
\[
T^*_\a f_\infty \colon \fm_\b/\fm_\b^2 \otimes_{k_\b} k_\a \to \fm_\a/\fm_\a^2,
\]
where $\b = f_\infty(\a)$, is an isomorphism.
\end{corollary}

\begin{proof}
Let $d = \dim_{\a(\eta)} X$, and let $f_1,\ldots,f_m \in k[x_1,\ldots,x_{n}]$ be generators of the ideal of $X$. 
Then $\Fitt^d \Om_{X/k}$ is generated by the $(n-d)\times (n-d)$-minors of the Jacobian matrix $J$ associated to 
$(f_1,\dots,f_m)$. After reordering, we may assume that $\ord_\a \Jac_X = \ord_\a \Delta$ where 
$\Delta$ is a minor of $J$ involving only $\frac{\partial f_i}{\partial x_j}$ for $j > d$. Consider now the morphism $f \colon X \to \A^d_k = \Spec k[x_1,\dots,x_d]$.
By construction, we have that 
\[
\ord_\a (\Fitt^d \Om_{X/k}) = \ord_\a (\Fitt^0 \Om_{X/\A^d_k}).
\] 
Furthermore, we have $\dim \Om_{X/k} \otimes k_{\a(\e)} = d$, hence
the above number is finite. 
In particular, this means that $f$ is unramified at $\a(\eta)$. 
Since $Y$ is smooth, it follows that $f$ is \'etale at $\a(\eta)$ (see for example \cite[(18.10.1)]{ega-iv-pt4}). 
The assertion then follows from \cref{t:differentials-cotangent-map}
and \cref{r:differentials-cotangent-map}.
\end{proof}

\begin{remark}
\label{r:cotangent-map-isom}
In the situation of \cref{c:cotangent-map-isom}, any morphism $f$ induced by a general linear projection $\A^n_k \to \A^d_k$ will do. Indeed, the condition $\ord_\a (\Fitt^0 \Om_{X/\A^d_k}) \geq \ord_\a (\Fitt^d \Om_{X/k})$ defines an open subset of $\Gr(d,n)$, and the argument in the proof 
of \cref{c:cotangent-map-isom} applies for any such choice of projection.
\end{remark}

\begin{remark}
The surjectivity of the cotangent map $T^*_\a f_\infty$ stated in part \eqref{i2:differentials-cotangent-map} of \cref{t:differentials-cotangent-map} was obtained in the proof of \cite[Proposition 4.5(iii)]{Reg09} in the special case where $\a$ is a stable point and $\charK k = 0$.
Note that in \cite[Proposition 4.5(iii)]{Reg09} it is also stated that if $\O_{X_\infty,\a}$ is Noetherian
then $f_\infty$ is unramified at $\a$, but the definition of unramified morphism adopted there is weaker than the standard one as it omits the requirement that the morphism is locally of finite type at $\a$
(cf.\ \cite[\href{https://stacks.math.columbia.edu/tag/02FM}{Tag 02FM}]{stacks-project}).
We will see later, in \cref{ex:not-formally-unramified}, that this last property can actually fail, 
hence $f_\infty$ can fail to be unramified at $\a$, even assuming that $\O_{X_\infty,\a}$ is Noetherian.
\end{remark}

Recall that if $f \colon X \to Y$ is a morphism of varieties over $k$, then $f$ is unramified at a point $x \in X$ if and only if the cotangent map $T^*_x f$ is surjective and the residue field extension $k_x/k_y$, where $y = f(x)$, is finite separable, see \cite[\href{https://stacks.math.columbia.edu/tag/02FM}{Tag 02FM}]{stacks-project}. 
At first glance, \cref{t:differentials-cotangent-map} 
seem to suggest that in the setting of part \eqref{i2:differentials-cotangent-map} of the theorem
the morphism $f_\infty$ is unramified at $\a$ (or at least formally unramified), 
but this is not true in general.
In fact, as we will discuss in \cref{s:finite-type}, in general $\Om_{X_\infty/Y_\infty}$ 
does not vanish at $\a$ if $f$ is ramified at $\a(0)$, 
even when $f$ is unramified at $\a(\e)$. It turns out however that
the completion of $\Om_{X_\infty/Y_\infty}$ with respect to the $\fm_\a$-adic topology on $\cO_{X_\infty,\a}$ does. Equivalently, the morphism $f_\infty$ is formally unramified at $\a$ in a weaker sense, which holds more generally whenever $f$ is formally unramified at $\a(\eta)$. 
To finish this section, we give a proof of this fact in \cref{l:unramified-form-unramified} and sketch how this can be used to obtain an alternative proof for \eqref{i2:differentials-cotangent-map} of \cref{t:differentials-cotangent-map} without making use of the results of \cite{dFD20}.

To start, let us recall the necessary definition from \cite{ega-iv-pt1}. 

\begin{definition}
    We say that a continuous map $R\to S$ between topological rings is \emph{formally unramified} if, for all discrete rings $A$ and ideals $J\subset A$ with $J^2 = 0$ fitting into a diagram with continuous maps
\begin{equation}
    \label{eq:formally-unramified-topological}
    \xymatrix{S \ar[r] \ar@{-->}[dr] & A/J \\
          R \ar[r] \ar[u] & A, \ar[u]}
\end{equation}
there exists at most one continuous $S \to A$ making the diagram commute.
\end{definition}

Equivalently, by \cite[(20.7.4)]{ega-iv-pt1} any continuous map $R \to S$ is formally unramified if and only if $\^\Om_{S/R} = 0$, where $\^\Om_{S/R}$ denotes the separated completion of $\Om_{S/R}$.

If $R$ and $S$ are rings endowed with the discrete topology, this agrees with the usual definition of formally unramified ring maps. However, applying it to the case of a local map $\varphi \colon (R,\fm) \to (S,\fn)$ with each side regarded with the adic topology translates to restricting to only those diagrams \eqref{eq:formally-unramified-topological} where $\varphi \colon R \to A$ satisfies $\varphi(\fm^n) = 0$ for some $n > 0$; and similarly for the map $S \to A/J$.

\begin{proposition}
\label{l:unramified-form-unramified}
Let $\a \in X_\infty$ and $f \colon  X\to Y$ be a morphism of schemes over a perfect field $k$ 
that is formally unramified at $\a(\eta)$. Then the induced map of local rings
\[
f_{\infty,\a}^\sharp \colon \cO_{Y_\infty,\b} \to \cO_{X_\infty,\a},
\]
where $\b = f_\infty(\a)$, is formally unramified as a map of topological rings.
\end{proposition}

\begin{proof}
Let us first show that, for any $\a \in X_\infty$ and for any continuous map of the form $\g \colon \cO_{X_\infty,\a} \to A$ where $A$ is a discrete $k$-algebra, the corresponding $A$-valued arc on $X$ induces a diagram
\[
\xymatrix{\cO_{X,\a(0)} \ar[r]^-{\g} \ar[d] & A[[t]] \ar[d] \\
          \cO_{X,\a(\eta)} \ar[r] & \LS{A}.}
\]
Indeed, take an element $f \in \cO_{X,\a(0)}$ that is not contained in the kernel of $\a^\sharp\colon \cO_{X,\a(0)} \to k_\a [[t]]$. Then $\a^\sharp(f) = f_e t^e + t^{e+1} ( \ldots)$ with $f_e \in k_\a^\times$. Writing $f^{(i)}$ for the $i$-th higher derivative of $f$ in $\cO_{X_\infty,\a}$, 
this means that $f^{(j)} \in \fm_\a$ for $j< e$ and $f^{(e)} \in (\cO_{X_\infty,\a})^\times$. Thus $\g(f^{(e)}) \in A^\times$ and, since $\g$ is continuous,  $\g(f^{(j)})$ is nilpotent for $j < e$. This implies that $\g(f) \in \LS{A}^\times$, see \cite[Proposition 1.3.1]{KV04}.

Now, any diagram of the form \eqref{eq:formally-unramified-topological} for $R = \cO_{Y_\infty,\b}$, $S = \cO_{X_\infty,\a}$ induces a diagram
\[
\xymatrix{\cO_{X,\a(\eta)} \ar[r] \ar@{-->}[dr] & \LS{A/J} \\
          \cO_{Y,\b(\eta)} \ar[u] \ar[r] & \LS{A}. \ar[u]}
\]
Since $f$ is formally unramified at $\a(\eta)$, there exists at most one diagonal arrow $\cO_{X,\a(\eta)} \to \LS{A}$. But $A[[t]] \to \LS{A}$ is injective and therefore there exists at most one lift $\cO_{X_\infty,\a} \to A$.
\end{proof}

Let us finally sketch how one can derive \eqref{i2:differentials-cotangent-map} of \cref{t:differentials-cotangent-map} from \cref{l:unramified-form-unramified}. Setting $R = \cO_{Y_\infty,\b}$ and $S = \cO_{X_\infty,\a}$, by \cite[(20.7.17)]{ega-iv-pt1} we have a sequence
\[
\xymatrix{\^\Om_{R/k} \cotimes_{\^R} \^{S} \ar[r]^-\rho & \^\Om_{S/k} \ar[r]^-\sigma & \^\Om_{S/R} }.
\]
Of note is that this sequence is not exact at $\^\Om_{S/k}$; instead, we just have that $\im \rho$ is dense in $\ker \sigma$. From \cref{l:unramified-form-unramified} it follows that $\rho$ has dense image. Writing $K = \cO_{X_\infty,\a}/\fm_\a$ we see that $\^\Om_{R/k} \cotimes_{\^R} K$ is discrete and isomorphic to $\Om_{R/k} \otimes_R K$; and similarly for $\^\Om_{S/k} \cotimes_{\^S} K$. Hence $\rho \cotimes K$ is surjective and agrees with the map $\Om_{R/k} \otimes_R K \to \Om_{S/k} \otimes_S K$. We can thus proceed as in the proof of \eqref{i2:differentials-cotangent-map} of \cref{t:differentials-cotangent-map}, using \eqref{i2:residue-field-extension} of \cref{t:residue-field-extension} for the finiteness of the field extension $k_\a / k_\b$.

%
%

\section{Finiteness of arc fibers}

\label{s:finite-fibers}

In this section we look at morphisms of varieties that have finite fibers and
ask whether the same property holds for the induced map on the level of arc spaces. 
As the notion of having finite fibers is more subtle for morphisms which are not of finite type, we start by introducing the following definition.

\begin{definition}
We say that a scheme $X$ is \emph{topologically finite} if its underlying topological space 
is finite discrete; in this case, we denote the cardinality of the latter by $|X|$. If $f \colon X \to Y$ is a morphism of schemes we say that $f$ has \emph{topologically finite fibers} if for every $y \in Y$ the fiber $f^{-1}(y)$ is topologically finite.
\end{definition}

Let us compare this with the more familiar notion of a quasi-finite morphism of schemes, which we recall here.

\begin{definition}
A morphism of schemes $f \colon X \to Y$ is \emph{quasi-finite} at a point $x\in X$ 
if it is locally of finite type at $x$ and $x$ is isolated in its fiber.
The morphism is \emph{quasi-finite} if it is of finite type and is quasi-finite at every point.
\end{definition}

\begin{remark}
Clearly a morphism is quasi-finite if and only if it is of finite type and has topologically finite fibers. Moreover, for a morphism of finite type, it is equivalent to require the fibers to be finite set-theoretically, topologically, or scheme-theoretically (cf.\ \cite[\href{https://stacks.math.columbia.edu/tag/02NG}{Tag 02NG}]{stacks-project}). This is no longer true, however, if the morphism is not of finite type. For example, for any discrete valuation ring $R$ over a field $k$ the fiber of $\Spec R \to \Spec k$ is obviously finite but not discrete. For an example of a topologically finite fiber that is not scheme-theoretically finite, consider any zero-dimensional $k$-algebra $A$ that is not Noetherian -- e.g., take $A = \Spec k[x_i\mid i \in \N]/(x_i\mid i \in \N)^2$, which is a fat point in $\A^\N$.
\end{remark}

Let us summarize the main results of this section in the following two theorems.

\begin{theorem}
\label{t:finite-fiber-1}
Let $f \colon X \to Y$ be a morphism of schemes over a perfect field $k$, and assume that $X$ is separated. Let $\a \in X_\infty$ and $\b = f_\infty(\a)$, and write $x = \a(\e)$ and $y = \b(\e)$. 
\begin{enumerate}
\item 
\label{i1:finite-fiber-1}
If $f^{-1}(y)_\red$ is quasi-finite at $x$ over $k_y$, then $\a$ is an isolated point of $f^{-1}_\infty(\b)$.
\item 
\label{i2:finite-fiber-1}
If $f$ is unramified at $x$, then $f^{-1}_\infty(\b)$ is reduced at $\a$ and $k_\a/k_\b$ is finite separable. 
\end{enumerate}
\end{theorem}

\begin{theorem}
\label{t:finite-fiber-2}
Let $f \colon X \to Y$ be a morphism of schemes over a perfect field $k$, and assume that $X$ is separated.
Let $\b\in Y_\infty$ and denote $y = \b(\e)$.
\begin{enumerate}
\item 
\label{i1:finite-fiber-2}
If $f^{-1}(y)_\red$ is finite over $k_y$, then $f_\infty^{-1}(\b)$ is topologically finite.
\item 
\label{i2:finite-fiber-2}
If $f$ is quasi-compact and unramified at each point of $f^{-1}(y)$, then $f_\infty^{-1}(\b)$ is finite over $k_\b$ and reduced. 
\end{enumerate}
\end{theorem}

\begin{proof}[Proof of \cref{t:finite-fiber-1,t:finite-fiber-2}]
First note that assertion \eqref{i1:finite-fiber-1} in \cref{t:finite-fiber-1}
follows from \cref{l:same-alg-closure}. 
The latter clearly implies that $\a$ is a closed point of $f_\infty^{-1}(\b)$.
To see that this point is in fact isolated, 
assume that it is a specialization of some other point $\a'$ in $f_\infty^{-1}(\b)$.
The point $x' = \a'(\e)$ specializes to $x$ and $f(x') = y$, hence $x' = x$ 
since $x$ is an isolated point of $f^{-1}(y)$. 
Then \cref{l:same-alg-closure}, applied this time to $\a'$, implies that $\a'$ is a closed
point of $f_\infty^{-1}(\b)$ and hence $\a' = \a$. 

The proof of 
the other assertions in these theorems proceeds as follows: we prove 
\eqref{i1:finite-fiber-2} of \cref{t:finite-fiber-2} in \cref{t:sep-deg-fin-fibers}, and   
\eqref{i2:finite-fiber-1} of \cref{t:finite-fiber-1} in \cref{t:unramified-reduced-arc-fiber}. 
To conclude, we just observe that if $f$ is quasi-compact and unramified at each point of $f^{-1}(y)$, then
$f^{-1}(y)$ is finite over $k_y$ and hence
\eqref{i2:finite-fiber-2} of \cref{t:finite-fiber-2} follows directly from the other assertions.
\end{proof}

Recall that for any algebraic field extension $K/L$ the separable closure $L_\sep$ 
of $L$ in $K$ is an intermediate extension $K/L_\sep/L$ such that $L_\sep/L$ is separable algebraic and $K/L_\sep$ is purely inseparable. The separable degree of a finite extension $K/L$ is then given by
\[
[K : L]_\sep := [L_\sep : L].
\]
We want to introduce a similar notion for schemes that are finite over a field.

\begin{definition}
\label{d:sep-deg}
Let $X$ be a finite scheme over a field $k$.
The \emph{separable degree} of $X$ is defined as 
\[
\sepdeg X := \sum_{x\in X} [k_x : k]_\sep.
\]
\end{definition}

\begin{remark}
\label{l:sep-deg-vs-deg}
For any finite $k$-scheme $X = \Spec A$ we have
\[
|X| \leq \sepdeg X \leq \deg X,
\]
where the degree of $X$ over $k$ is given by $\deg X := \dim_k A$. 
\end{remark}


\begin{lemma}
\label{l:sep-deg}
Given a finite field extension $K/L$ of separable degree $d$ and an arbitrary extension $K'/L$, there are at most $d$ embeddings of $K$ in $K'$ over $L$.
\end{lemma}

\begin{proof}
Consider the relative separable closure $L \subset L_\sep \subset K$. By the primitive element theorem, $L_\sep = L[z] = L[Z]/(P(Z))$, where $z \in L_\sep$ is algebraic with minimal polynomial $P(Z)$ over $L$. Then an embedding $L_\sep \hookrightarrow K'$ over $L$ is determined by the image of $z$, which must be a root of $P(Z)$ in $K'$. 
Since $\deg(P(Z)) = [K : L_\sep] = d$, there are at most $d$ such roots.

We will conclude if we show that given a fixed embedding $L_\sep \subset K'$, there is at most one embedding $K \hookrightarrow K'$ extending $L_\sep \subset K'$. For this, notice that, since $K/L_\sep$ is purely inseparable, every element $w \in K$ has a minimal polynomial over $L_\sep$ that has only one distinct root. This means that the image of $w$ in $K'$, if it exists, is unique.
\end{proof}


\begin{proposition}
\label{t:sep-deg-fin-fibers}
Let $f \colon X \to Y$ be a morphism of schemes over a field $k$, and assume that 
$X$ is separated. Let $\b \in Y_\infty$ and let $y = \b(\e) \in Y$. If $f^{-1}(y)_\red$ is finite over $k_y$, then $f_\infty^{-1}(\b)$ is topologically finite of cardinality at most $\sepdeg f^{-1}(y)$.
\end{proposition}

\begin{proof}
Write $f^{-1}(y)_\red = \Spec (K_1 \times \cdots \times K_n)$ where each $K_i/k_y$ is a finite field extension and set $d_i = [K_i : k_y]_\sep$, so that $\sepdeg f^{-1}(y) = \sum_{i=1}^n d_i$. Let $x \in f^{-1}(y)$ with $k_{x} = K_i$. Fix an algebraic closure $\ov{k_\b}$ of $k_\b$ and let $\a \in f^{-1}_\infty(\b)$ with $\a(\eta) = x$. By \cref{l:same-alg-closure} there exists an embedding $\tau_\a: k_\a \to \ov{k_\b}$. Note that the composition $\Spec \ov{k_\b} \to \Spec k_\a \to X_\infty$ defines the same point independently of the choice of embedding $\tau_\a$. Thus any diagram of the form
\begin{equation}
\label{eq:geom-fiber}
\xymatrix{X_\infty \ar[d]_-{f_\infty} & \\
          Y_\infty & \Spec \ov{k_\b} \ar[l]^-{\ov{\b}} \ar[lu]}
\end{equation}
determines a point in the fiber $f_\infty^{-1}(\b)$, and conversely every point $\a$ in $f_\infty^{-1}(\b)$
with $\a(\e) = x$ arises in this way.  
The diagram \eqref{eq:geom-fiber} yields a diagram of field extensions
\begin{equation}
\label{eq:geom-fiber-fields}
\xymatrix{k_x \ar[rd] & \\
          k_y \ar[r]_-{\ov{\b}^\sharp_\eta} \ar[u]^-{f_x^\sharp} & \LS{\ov{k_\b}}.}
\end{equation}  
As $X$ is separated, any such diagram gives rise to at most one $\ov{k_\b}$-valued arc on $X$, which in turn gives an $\a \in f^{-1}_\infty(\b)$. Thus it suffices to show that there exist at most $d_i$ intermediate extensions of the form \eqref{eq:geom-fiber-fields}. 

By hypothesis the extension $f_x^\sharp$ is finite with separable degree $d_i$, and by \cref{l:sep-deg} we see that there are at most $d_i$ possible embeddings $k_x \to \LS{\ov{k_\b}}$ fitting into \eqref{eq:geom-fiber-fields}. This shows that there are at most $d_i$ points $\a \in X_\infty$ such that $f_\infty(\a) = \b$ and $\a(\e) = x$. Ranging among all possible $x \in f^{-1}(y)$ (i.e., among the $K_i$) we see that the fiber $f_\infty^{-1}(\b)$ is finite with cardinality at most $d$.

To conclude, notice that $f_\infty^{-1}(\b)$ is a finite topological space in which, by \cref{l:same-alg-closure}, all points are closed. This means that $f_\infty^{-1}(\b)$ is discrete.
\end{proof}


\begin{proposition}
\label{t:unramified-reduced-arc-fiber}
Let $f \colon X \to Y$ be a morphism of schemes over a perfect field $k$. 
Let $\a \in X_\infty$ and $\b = f_\infty(\a)$, 
and assume that $f$ is unramified at $\a(\eta)$. Then $\a$ is reduced in $f_\infty^{-1}(\b)$
and $k_\a/k_\b$ is finite separable.
\end{proposition}

\begin{proof}

Write $x = \a(\eta)$, and $y = \b(\eta)$. The fact that $f$ is unramified at $\a(\eta)$ implies that the extension of residue fields $k_x/k_y$ is finite separable, 
hence $k_\a/k_\b$ is finite separable by \cref{t:residue-field-extension}.
By the same argument used in the proof of assertion \eqref{i1:finite-fiber-1} in \cref{t:finite-fiber-1}, 
we see by \cref{l:same-alg-closure} that $\a$ is an isolated point in its fiber. Thus we may pass to suitable affine open sets and assume that $X = \Spec S$ and $Y = \Spec R$. As $f$ is locally of finite type at $\a(\eta)$, we can write
\[
S \isom R[x_1,\ldots,x_n]/(f_i \mid i\in I).
\]

We first show that $n \leq |I|$. 
Writing
\[
\Om_{S/R} = \langle dx_1,\ldots,dx_n \rangle/ \langle \sum_j \frac{\partial f_i}{\partial x_j} dx_j \mid i\in I \rangle,
\]
we have a free presentation of $\Om_{S/R}$ of the form
\[
\xymatrix@C=15pt{S^I \ar[r]^-\tau & S^n \ar[r] & \Om_{S/R} \ar[r] & 0.}
\]
Since $f$ is unramified at $\a(\eta)$ we have that $\Om_{S/R} \otimes_S \LS{k_\a} = 0$.
Hence $\tau \otimes \id_{\LS{k_\a}}$ is a surjection, and thus $n \leq |I|$.

The next step is to reduce to the case where $|I|=n$.
Since $\Om_{S/R} \otimes \LS{k_\a} = 0$, there exist $i_1,\ldots,i_n \in I$ such that the image of 
$\det \left( \frac{\partial f_{i_e}}{\partial x_j}\right)_{e,j\leq n}$ in $\LS{k_\a}$ is invertible. 
The factorization
\[
\xymatrix@C=15pt{ R \ar[r] &  R[x_1,\dots,x_n]/(f_{i_1},\ldots,f_{i_n}) \ar[r] & S }
\]
corresponds to a factorization $X \to X' \to Y$ where $X\to X'$ is
a closed immersion and the induced map $f' \colon X' \to Y$ is unramified at $\a'(\eta)$, where $\a'$ is the image of $\a$ in $X'_\infty$. By \cref{l:same-alg-closure}, $\a'$ is an isolated point of its fiber. Then it clearly suffices to prove that $\a'$ is reduced in its fiber, hence we may assume without loss of 
generality that $|I|=n$. 

Let $\ov{k_\b}$ be an algebraic closure of $\k_\b$. We look at the fiber $F = X_\infty \times_{Y_\infty} k_\b$ and
the geometric fiber $\ov{F} = X_\infty \times_{Y_\infty} \Spec \ov{k_\b}$ over $\b$. We aim to show that $\ov{F}$ is reduced at every point $\ov{\a}$ lying over $\a$. By \cite[\href{https://stacks.math.columbia.edu/tag/035W}{Tag 035W}]{stacks-project}, this will imply that the scheme $\ov{F}_\a = \Spec \cO_{X_\infty,\a} \otimes \ov{k_\b}$ is reduced, thus $F$ is geometrically reduced at $\a$ and hence in particular reduced at $\a$. 

First note that the residue field of each $\ov{\a}$ is just $\ov{k_\b}$ and thus, by \cite[\href{https://stacks.math.columbia.edu/tag/01TE}{Tag 01TE}]{stacks-project} $\ov{\a}$ is isolated in $\ov{F}_\a$. In particular, the local ring $\cO_{\ov{F},\ov{\a}}$ is of dimension $0$. Let $\cA$ be the category whose objects are local $\ov{k_\b}$-algebras $(A,\fm)$ with $\fm = \Nil(A)$ and $A/\fm = \ov{k_\b}$ and whose morphisms are local ring maps. The geometric fiber over $\b$ at $\ov{\a}$ is determined by its $A$-points, which in turn correspond to diagrams of the form
\[
\xymatrix{\Spec A[[t]] \ar[d] \ar[r]^-{\widetilde {\a}} & X \ar[d]^-f \\
          \Spec \ov{k_\b}[[t]] \ar[r]_-{\ov{\b}} & Y,}
\]
where $A \in \cA$ and precomposing $\widetilde{\a}$ with $\Spec \ov{k_\b}[[t]] \to \Spec A[[t]]$ gives $\ov{\a}$. In order to prove that $\ov{F}$ is reduced at $\ov{\a}$ it suffices to show that for any $A \in \cA$ there exists at most one morphism $\widetilde \a$ making the above diagram commute; as $\ov{k_\b}$ is clearly the initial object in $\cA$ we are done. Note that for every $A$ there exists such a morphism $\widetilde a$ given by $\widetilde \a = \ov{\a} \circ \iota$, where $\iota \colon  \Spec A[[t]] \to \Spec \ov{k_\b}[[t]]$ is the natural inclusion.

Write $f = (f_1,\ldots,f_n)$ and $x = (x_1,\ldots,x_n)$. Then $\ov{\a}$ corresponds to
\[
\ov{\a}(t) = (\ov{\a}_1(t),\dots,\ov{\a}_n(t)) \in \ov{k_\b}[[t]]^n,
\]
satisfying $f(\ov{\a}(t)) = 0$. 
Since $f$ is unramified at $\a(\eta)$, we have that $\det M(\ov{\a}(t)) \neq 0$, where $M = \left(\frac{\partial f_i}{\partial x_j}\right)_{i,j \leq n}$. Write $\det M(\ov{\a}(t)) = t^d u(t)$ with $d\geq 0$ and $u(t) \in \ov{k_\b}[[t]]$ invertible. 
Let $(A,\fm) \in \cA$, then an $A$-point $\widetilde \a$ of the fiber at $\ov{\a}$ is given by
\[
\widetilde \a(t) = \ov{\a}(t) + v(t) \in A[[t]]^n, \: v(t) \in \fm[[t]]^n,
\]
satisfying $f(\widetilde \a(t)) = 0$. Taking Taylor expansion we get
\[
0 = \underbrace{f(\ov{\a})}_{=0} + M(\ov{\a}(t)) \cdot v(t) + N(t) \cdot v(t),
\]
where $N(t)$ has its coefficients in $\fm[[t]]$. Multiplying with the adjugate matrix of $M(\ov{\a}(t))$ we have
\[
0 = (t^d u(t) \id_n + \widetilde N(t)) \cdot v(t).
\]
By \cref{l:matrix-nilpotent} we see that $(t^d u(t) \id_n + \widetilde N(t))$ is invertible considered as a matrix with coefficients in $\LS{A}$. As $A[[t]] \to \LS{A}$ is injective we have $v(t) = 0$. This finishes the proof.
\end{proof}

\begin{lemma}
\label{l:matrix-nilpotent}
Let $(A,\fm)$ be a $k$-algebra with $A/\fm = k$ and $\fm = \Nil(A)$. Consider a matrix $M \in M_{n\times n}(A[[t]])$ of the form $M = t^d u(t) \id_n + N$ where $u(t) \in A[[t]]$ is invertible and $N \in M_{n\times n}(\fm[[t]])$. Then there exists $M' \in M_{n\times n}(A[[t]])$ such that $M' \cdot M = t^e U(t)$ for some $U(t) \in \GL_n(A[[t]])$.
\end{lemma}

\begin{proof}
We write $N = N_1 + N_2$ where each entry of $N_1$ has order $\geq d+1$ and each entry of $N_2$ is a polynomial of degree at most $d$. In particular, each entry of $N_2$ is nilpotent and by the Cayley--Hamilton theorem $N_2$ is nilpotent itself. Write 
$t^d u(t) \id_n + N_1 = t^d U'(t)$ where $U'(t)$ is invertible modulo $\fm$, 
hence invertible. Then there exists $M' \in M_{n\times n}(A[[t]])$ such that $M' \cdot M = (t^d U')^e$.
\end{proof}

In the following examples we demonstrate the statement of
\cref{t:unramified-reduced-arc-fiber} by explicitly showing nontrivial nilpotent
elements in the local ring $\O_{X_\infty,\a}$ 
vanishing along the fiber of $f_\infty$ at $\a$. As we will see later in \cref{p:not-reduced-not-noetherian}, this implies that $\O_{X_\infty,\a}$ is not Noetherian.

\begin{example}
\label{eg:node}
Let $X \subset \A^2_k$ be the node defined by the equation $h = xy = 0$,
and let $\a \in X_\infty$ be an arc of the form $\a = (-a(t),0)$ where  $a(t) = \sum_{m\ge 1} a_m t^m$.
We assume that $a_m \ne 0$ for all $m \ge 1$. 
As usual, we write $(\A^2_k)_\infty = \Spec k[x_n,y_n \mid n \ge 0]$
where $x_n$ and $y_n$ denote the $n$-th higher derivatives of $x$ and $y$.
Consider the element $x_0y_1$ and its image $g$ in the local ring $\O_{X_\infty,\a}$.
It is easy to check that $g^2 = 0$ (cf.\ \cite[Example~7]{Seb11}),
and we claim that $g \ne 0$;
\cite[Example~7]{Seb11} shows that the element $x_0y_1$ does not vanish 
on $X_\infty$, but it is not clear how to deduce from there that the element
remains nonzero after localization. 
In fact, we claim that all higher derivatives $g^{(i)}$ of $g$
are nonzero in $\O_{X_\infty,\a}$. Note that all these elements are nilpotent 
because they are the coefficients in the image of $g$ under the map
$\O_{X_\infty,\a} \to \O_{X_\infty,\a}[[t]]$ obtained by localizing 
at $\a$ the prolongation $\O_{X_\infty} \to \O_{X_\infty}[[t]]$ 
of the universal arc $\cO_{X} \to \O_{X_\infty}[[t]]$.

We prove that $g^{(i)} \ne 0$ for all $i \ge 0$ using a similar test ring as in 
\cite[Example~9.6]{dFD20}. Specifically, let 
$\Gamma = \mathbb Z \oplus \epsilon \mathbb Z$ with the lexicographic order, 
and let $R \subset \k_\a(r,s)$ be the rank 2 valuation ring with value group $\Gamma$
associated to the monomial valuation $v$ defined by $v(r) = \epsilon$
and $v(s) = 1$. Consider the map $k[x,y]/(xy) \to (R/(rs))[[t]]$ defined by 
\[
x \mapsto r - a(t), \quad
y \mapsto s + \frac s r\, a(t) + \frac s{r^2} \, a(t)^2 + \frac s{r^3}\, a(t)^3 + \dots.
\]
The induced map $\O_{X_\infty,\a} \to R/(rs)$ sends 
$g^{(i)} \mapsto sa_{i+1}$, something that can be easily checked by observing that 
$(x_0y_1)^{(i)} = h^{(i+1)} - x_{i+1}y_0$ 
(where $h^{(i+1)}$ is the $(i+1)$-th higher derivative of $h$),
hence $g^{(i)} = - x_{i+1}y_0$ in $\O_{X_\infty,\a}$. This shows that $g^{(i)} \ne 0$.

Now, let $f \colon X \to Y = \A^1_k$ be the projection given by $z = x - y$,
and let $\b = f_\infty(\a) \in Y_\infty$. Explicitly, $\b$ is given by $z = -a(t)$.
Using the equations $z_0 = h^{(1)} = 0$, which hold on $f_\infty^{-1}(\b)$, we see that 
$x_0(y_1 + x_1) = 0$ on the fiber, and since
$y_1 + x_1$ does not vanish at $\a$, this implies that $x_0 = 0$
in $S = \O_{X_\infty,\a}/\fm_\b\.\O_{X_\infty,\a}$. 
As $g^{(i)} = - x_{i+1}y_0$ in $\O_{X_\infty,\a}$ and $y_0 = x_0$ on $f_\infty^{-1}(\b)$, we conclude that
the image of $g^{(i)}$ in $S$ is zero. 
\end{example}

\begin{example}
\label{eg:cusp}
Let $X$ be the cuspidal plane curve singularity given by $h = x^3-y^2 = 0$ and assume that $\charK k\neq 2,3$.
Let $\n \colon X' \to X$ be the normalization and $p \in X'$ the preimage of the cusp, 
and let $\a \in X_\infty$ be the image of the generic arc $\g$ on $X'$ with order of contact 1 with $p$. 
It is claimed in \cite[Remark 3.16]{Reg09} that the element 
$g \in \cO_{X_\infty,\a}$ defined by $2x_0y_1 - 3x_1y_0$, 
and all its higher derivatives $g^{(i)}$, are nontrivial nilpotent elements.
A proof that $2x_0y_1 - 3x_1y_0$ does not vanish globally on $X_\infty$ is given in 
\cite[Example~8]{Seb11}, but this does not imply directly the nonvanishing of
its image $g$ in the localization. 

We give a proof that $g \ne 0$ by reducing to a computation similar to the one carried out in
\cref{eg:node}. 
Consider the map $k[x,y] \to k[u,v]$ given by 
$x \mapsto u^2(v+1)$, $y \mapsto u^3(v+1)^2$.
This corresponds to an affine chart of the minimal embedded log resolution of the cusp
centered at the intersection of the proper transform of the cusp (which is
given by $v=0$ and can be identified with the normalization $X'$) 
with the exceptional divisor extracted by the third blow-up (given by $u=0$). 
Under this map, $y^2-x^3 \mapsto u^6(v+1)^3v$. In particular, if $V$
is the scheme defined in this chart by $u^6v = 0$, then we have a morphism $\m \colon V \to X$.
Note that $X' \subset V$. 
The arc $\g$ can be written as $\g = (-a(t),0)$ in the coordinates $(u,v)$, 
with $a(t) = \sum_{m\ge 1} a_mt^m$. We regard $\g$ as a point in $V_\infty$.
Let $R$ be the valuation ring defined as in \cref{eg:node}, and consider the map
$k[u,v]/(u^6v) \to (R/(r^6s))[[t]]$ given by
\[
u \mapsto r - a(t), \quad
v \mapsto s \Big(1 + \frac 1 r\, a(t) + \frac 1{r^2} \, a(t)^2 + \frac 1{r^3}\, a(t)^3 + \dots\Big)^6.
\]
Then
\[
2x_0y_1 - 3x_1y_0 \mapsto u_0^5(v_0+1)^2v_1 \mapsto 6 r^4(s+1)^2s a_1  
\]
which is nonzero in $R/(r^6s)$. 
This implies that the image of $g$ under the induced map
$\O_{X_\infty,\a} \to R/(r^6s)$
is nonzero, hence $g \ne 0$.  

Now, consider the morphism $f \colon X \to Y = \A^1_k$ induced by the
projection $(x,y) \mapsto x$.
We look at the vanishing of the elements $g^{(i)}$
on the fiber of $f_\infty$ at $\a$, which we expect by
\cref{t:unramified-reduced-arc-fiber}. 
We carry out the computation showing the vanishing of
the restriction of $g^{(i)}$ to the fiber at $\a$ for $i \leq 5$.
To that avail, let $h^{(i)}$ denote the $i$-th higher derivative of $h$, 
and write $\ov{h^{(i)}}$ for its image in the fiber ring
over the ideal $(x_0,x_1)$. Similarly, let $\ov{g^{(i)}}$ the image of $g^{(i)}$ in the
quotient ring $S = \cO_{X_\infty,\a}/\fm_\b \cdot \cO_{X_\infty,\a}$. By
computing $\ov{h^{(i)}}$ for $i\leq 6$, we see that $y_0 = y_1 = y_2 = 0$ in $S$.
Thus $\ov{g^{(i)}} = 0$ for $i\leq 4$. For $i = 5$ we have
\[
    \ov{g^{(5)}} = 2 x_2 y_4 - 3x_3 y_3.
\]
Note that $\ov{h^{(6)}} = x_2^3 - y_3^2$ and $\ov{h^{(7)}} = 3x_2^2x_3 -2y_3y_4$. Then we have
\[
    \ov{g^{(5)}} = 3 y_3^{-1} x_2^3 x_3 - 3x_3 y_3 = 0
\]
by the vanishing of $\ov{h^{(7)}}$ and $\ov{h^{(6)}}$, respectively.
In a similar fashion, one can use the vanishing of $\ov{h^{(8)}}$ to conclude that $\ov{g^{(6)}} = 0$ 
in $S$, and so on.
\end{example}

\begin{remark}
\label{r:arc-fibers-vs-greeenberg-schemes}
The results of this section can be equivalently regarded as results about Greenberg schemes.
On the one hand, every arc fiber is a Greenberg scheme, since given $f\colon X \to Y$ and $\b\in Y_\infty$ 
we have a canonical isomorphism of $f_\infty^{-1}(\b)$ with the Greenberg scheme of $X \times_Y \Spec k_\b[[t]]$
over $\Spec k_\b[[t]]$.
On the other hand, if $R$ is an equicharacteristic complete discrete valuation ring and $\cX$ is a scheme over $R$, then the Greenberg scheme of $\cX$ is isomorphic to $\cX_\infty \times_{(\Spec R)_\infty} \Spec k_\g$ where
the arc spaces are taken over a coefficient field $k$ of $R$ and $\g \in (\Spec R)_\infty$ is the arc 
given by a choice of uniformizer $t$ for $R$. 
\end{remark}

%
%

\section{Arc fibers of quasi-finite morphisms}

\label{s:bounds-quasifinite-morphisms}

Here we look more closely at the case of quasi-finite morphisms. 
Our first theorem shows that the arc fibers
of quasi-finite morphisms are bounded and
scheme-theoretically finite away from the arc space of the ramification locus. 
It can be viewed as a global analogue of the results of \cref{s:finite-fibers}.

\begin{theorem}
\label{t:finite-fiber}
Let $f \colon  X\to Y$ be a quasi-finite morphism of schemes over a perfect field $k$, 
and assume that $X$ is separated and quasi-compact.
Let $R := \Supp \Om_{X/Y}$ denote the ramification locus of $f$. 
Then the induced morphism $f_\infty\colon X_\infty \to Y_\infty$ satisfies the following properties:
\begin{enumerate}
\item 
\label{i1:finite-fiber}
$f_\infty$ has topologically finite fibers of bounded cardinality. 
\item
\label{i2:finite-fiber}
The restriction of $f_\infty$ 
to $X_\infty \setminus R_\infty$ has finite reduced fibers.
\end{enumerate}
\end{theorem}

\begin{proof}
   By \cite[\href{https://stacks.math.columbia.edu/tag/03JA}{Tag 03JA}]{stacks-project} and \cref{l:sep-deg-vs-deg}, for any quasi-finite morphism $X \to Y$ with $X$ quasi-compact there exists $d \geq 0$ such that $\sepdeg f^{-1}(y) \leq d$ for all $y \in Y$. Then \eqref{i1:finite-fiber} follows from \cref{t:sep-deg-fin-fibers}. The second assertion follows directly from \eqref{i2:finite-fiber-1} of \cref{t:finite-fiber-1}.
\end{proof}

We now focus on the case of finite morphisms between varieties over an algebraically closed field. By \cref{t:finite-fiber} the cardinality of the arc fibers of such morphisms are bounded, and the goal here is to provide an effective bound in this special case.

Let us first recall the following well-known result which, for a finite morphism $f \colon X \to Y$ of varieties over an algebraically closed field, provides an explicit bound for the cardinality of closed fibers when $Y$ is normal.

\begin{theorem}[\protect{\cite[Chapter 2, Section 6.3, Theorem 2.28]{Sha}}]
\label{t:shafarevich}
Let $f \colon X \to Y$ be a finite surjective morphism of varieties over an algebraically closed field $k$, 
and assume that $Y$ is normal. Then $|f^{-1}(y)| \leq \deg f$ for every $y \in Y(k)$.
\end{theorem}

Our next theorem can be viewed as an analogue of \cref{t:shafarevich} for the induced map 
at the level of arc spaces $f_\infty \colon X_\infty \to Y_\infty$.
Notice that replacing the degree of $f$ with its separable degree provides a sharper 
bound on the cardinality of the fibers. 

\begin{theorem}
\label{t:finite-fiber-deg}
Let $f \colon X \to Y$ be a finite surjective morphism between varieties over an algebraically closed field $k$, 
and assume that $Y$ is normal. Then $|f^{-1}_\infty(\beta)| \leq \sepdeg f$ for every $\b \in Y_\infty$.
\end{theorem}

The main new ingredient in the proof of \cref{t:finite-fiber-deg} is to show that the bound established in \cref{t:shafarevich} works for all (possibly non-closed) fibers of $f$. In the case where $\charK k > 0$ we need the following auxiliary construction in order to replace degree with separable degree.

\begin{lemma}
\label{l:scheme-separable-cl}
Let $f \colon X \to Y$ be a finite surjective morphism between varieties over a field $k$ and assume that $Y$ is normal. Let $\n \colon X' \to X$ be the normalization of $X$. Then there exists a diagram of varieties
\[
\xymatrix{Y & Y' \ar[l]_-g\\
             X \ar[u]^-f & X' \ar[l]^-{\n} \ar[u]_-h}
\]
such that $g$ is finite surjective with $k(Y) \subset k(Y')$ separable and $\deg g = \sepdeg f$, and
$h$ is a universal homeomorphism.
\end{lemma}

\begin{proof}
We will only prove the assertion in case $X = \Spec R$ and $Y = \Spec S$ are
affine; the general case follows much in the same way. Write $L = \Quot S$ and
$K = \Quot R$ and let $L_\sep$ be the separable closure of $L$ in $K$. Let $S'$
be the integral closure of $S$ inside $L_\sep$ and $Y' = \Spec S'$.
By \cite[\href{https://stacks.math.columbia.edu/tag/032L}{Tag
032L}]{stacks-project}, $g \colon Y' \to Y$ is a finite surjective morphism of
varieties and clearly $k(Y') = L_\sep$. Moreover, $h$ is universally injective
(see for example \cite[Exercise 5.3.9]{QL02}) and thus a universal
homeomorphism.
\end{proof}

We obtain the following improvement of \cref{t:shafarevich}. 

\begin{proposition}
\label{l:finding-enough-points}
Let $f \colon X \to Y$ be a finite surjective morphism of varieties over an algebraically closed field, and assume that $Y$ is normal.
\begin{enumerate}
\item 
\label{i1:finding-enough-points}
For every $y \in Y(k)$ we have $|f^{-1}(y)| \leq \sepdeg f$.
\item 
\label{i2:finding-enough-points}
There exist nonempty open sets $V \subset Y$ and $U \subset f^{-1}(V) \subset X$ such that for the restriction $f_U\colon U \to V$ we have $|f_U^{-1}(y)| = \sepdeg f$ for every $y \in V(k)$.
\end{enumerate}
\end{proposition}

\begin{proof}
We keep the notation of \cref{l:scheme-separable-cl}. 
For short, let $d = \sepdeg f$. To prove \eqref{i1:finding-enough-points}, by \cref{t:shafarevich} we have that $|g^{-1}(y)| \leq d$ for every $y \in Y'(k)$ and hence the same holds for the composition $gh$. Since the normalization $\n\colon X' \to X$ is finite surjective we get the assertion.

For \eqref{i2:finding-enough-points}, note that, since $\n$ is birational and $h$ is a homeomorphism it suffices to prove the statement for $g$. Hence we may assume that $k(Y) \subset k(X)$ is separable of degree $d$. Thus $f$ is generically unramified and there exists a nonempty open set $U'' \subset X$ such that $U'' \to Y$ is unramified. Moreover, since the rank of $f_*\cO_X$ is upper semicontinuous, we can find a nonempty open set $V \subset Y$ where $f_*\cO_X$ is of rank $d$. Set $U' = f^{-1}(V)$ and $U = U' \cap U''$. Now consider the fiber $f^{-1}(y) \cap U$ of $y\in Y(k)$ with respect to $U \to V$. As $U\to V$ is unramified the fiber decomposes into a product of copies of $k$, which implies that $|f^{-1}(y) \cap U| = d$.
\end{proof}

Now we can prove that the separable degree of every fiber of $f$ is bounded by the 
separable degree of the map, provided of course that $Y$ is normal.

\begin{corollary}
\label{p:bound-fiber-degree}
Let $f \colon X\to Y$ be a finite surjective morphism of varieties over an algebraically closed field $k$, 
and assume that $Y$ is normal. Then $\sepdeg f^{-1}(y) \leq \sepdeg f$ for every $y\in Y$.
\end{corollary}

\begin{proof}
Write $f^{-1}(y) = \{x_1,\ldots,x_r\}$ and let $d_i = [k(x_i) : k(y)]_\sep$. Furthermore write $Z_i = \overline{\{x_i\}}$ and $W = \overline{\{y\}}$. For each $i$ the restriction $f_i \colon Z_i \to W$ of $f$ is finite of separable degree $d_i$. By part \eqref{i2:finding-enough-points} of \cref{l:finding-enough-points} there exist nonempty open sets $U_i \subset Z_i$ and $W_i \subset W$ such that for $f_i \colon U_i \to W_i$ and $y \in W_i(k)$ we have $|f_i^{-1}(y)| \leq d_i$. Since $f$ is finite, we have that the intersection $Z_i \cap Z_j$ is a proper closed subset of $Z_i$ for $j\neq i$. Therefore the set $U'_i = U_i \setminus \bigcup_{j\neq i} U_j$ is dense in $U_i$ and so is the image of the set $U' = \bigcup_i U'_i$ in $W$. The set $W' = f(U') \cap \bigcap_i W_i$ is nonempty and constructible and thus contains a $k$-rational point $y$. Therefore the fiber $f^{-1}(y)$ contains at least $\sum_i d_i$ points and hence, by the first assertion of \cref{l:finding-enough-points}, we have $\sum_i d_i \leq d$.
\end{proof}

\begin{proof}[Proof of \cref{t:finite-fiber-deg}]
It follows from \cref{p:bound-fiber-degree} and \cref{t:sep-deg-fin-fibers}.
\end{proof}

\begin{remark}
    As a related result we want to mention \cite[Lemma~4.2]{BPR18}, which states that 
if $f \colon X \to Y$ is a finite and radicial morphism of varieties over a perfect field $k$
of positive characteristic, then $f_\infty$ is integral and radicial. We thank Devlin Mallory for bringing this result to our attention. It is not known whether the analogous property in characteristic $0$ holds, see \cite[Remark~4.3]{BPR18}.
\end{remark}

Uniform bounds on the cardinality of the fibers of $f_\infty$ also hold for
more general morphisms $f$ if we restrict the attention to fibers over
arcs $\b$ whose image is dense in $Y$, that is, such that $\b(\e)$ is the generic point of $Y$.
Let us first recall the following definition. 

\begin{definition}
A morphism of varieties $f \colon X \to Y$ is said to be \emph{generically finite} if $f$ is 
dominant and quasi-finite at the generic point of $X$. In this case, in analogy to \cref{d:sep-deg}, 
we define the \emph{separable degree} of $f$ as
\[
\sepdeg f := [k(Y) : k(X)]_\sep.
\]
\end{definition}

\begin{proposition}
\label{c:finite-fiber-fat}
Let $f \colon X \to Y$ be a generically finite morphism of varieties
over a field $k$.
Then $|f_\infty^{-1}(\b)| \le \sepdeg f$ for every $\b \in Y_\infty$ such that
$\b(\e)$ is the generic point of $Y$. 
\end{proposition}

\begin{proof}
It follows immediately by \cref{t:sep-deg-fin-fibers}. 
\end{proof}

\section{The ramification locus of $f_\infty$}

\label{s:finite-type}

In this section we aim to understand the ramification locus of $f_\infty$,                  
characterize where $f_\infty$ is locally of finite type, 
and determine under which conditions 
$f_\infty$ is globally a morphism of finite type.

We know by the base-change property of arc spaces along \'etale morphisms
that if $f$ is \'etale then so is $f_\infty$, and in particular $f_\infty$ is locally of finite type.
By contrast, even in the very simple case of a ramified double cover $f \colon \A^1_k \to \A^1_k$
the morphism $f_\infty$ presents interesting pathologies. 

\begin{example}
\label{ex:not-formally-unramified}
Let $k$ be a field of characteristic $\neq 2$ and $X = Y = \A^1_k$. Consider the morphism $f \colon X \to Y$ given by $y = x^2$ and let $\a$ be the generic arc on $X$ with order of contact $1$ with the origin. Writing 
$X_\infty = \Spec k[x_0,x_1,x_2,\dots]$ where $x_i$ is the $i$-th higher derivative of the coordinate $x$ on $X$, 
we see that $\a$ corresponds to the prime ideal $(x_0)$ of $k[x_0,x_1,x_2,\ldots]$.
The image $\b = f_\infty(\a)$ is the generic arc on $Y$ with order of contact $2$ with the origin; 
writing $Y_\infty = \Spec k[y_0,y_1,y_2,\dots]$, $\b$ corresponds to the prime $(y_0,y_1)$.
Note that the map $f_\infty^\sharp \colon \cO_{Y_\infty} \to \cO_{X_\infty}$ is given by the higher derivatives of $y = x^2$, that is,
\begin{equation}
\label{eq:not-formally-unramified}
y_0 = x_0^2,\quad y_1 = 2x_0x_1,\quad y_2 = x_1^2 + 2x_0x_2, \quad \ldots
\end{equation}
Using this presentation, we can compute the module of relative differentials $\Om_{X_\infty/Y_\infty}$ and get
\[
\Om_{X_\infty/Y_\infty} \isom \langle dx_i \mid i\in \Z_{\ge 0} \rangle / \langle x_0 dx_0, x_1 dx_0 + x_0 dx_1, \ldots \rangle.
\]
We want to compute the stalk of $\Om_{X_\infty/Y_\infty}$ over $\a$. 
Denote for short $S := \cO_{X_\infty,\a}$ and $M := \Om_{X_\infty/Y_\infty} \otimes S$, and define
\[
    M_n := \langle dx_i \mid i=0,\ldots,n \rangle / \langle \sum_{j=0}^i x_j dx_{i-j} \mid i=0,\ldots,n \rangle
\]
Note that $S$ is a discrete valuation ring with uniformizer $x_0$.
Since $x_j \in S$ is invertible for $j>0$, we have $M_n \isom \langle dx_n \rangle / \langle x_0^{n+1} dx_n \rangle$. Moreover, by the structure theorem for modules over principal ideal domains, $M$ is the colimit of the system
\[
    \xymatrix@C=30pt{M_0 \isom S/(x_0) \ar[r]^-{\cdot u_1 x_0 } & M_1 \isom S/(x_0^2) \ar[r]^-{\cdot u_2 x_0 } & \;\ldots\; \ar[r]^-{\cdot u_{n-1} x_0} & M_n \isom S/(x_0^{n+1}) \ar[r]^-{\cdot u_{n} x_0} & \;\ldots }
\]
where $u_j \in S$ are units. This shows that $dx_0 \neq 0$ in $M_n$ for all $n\in \N$ and so $dx_0 \neq 0$ in $M$.
Therefore $f_\infty$ is not unramified at $\a$ (and in fact it is not formally unramified 
in any neighborhood of $\a$ \cite[\href{https://stacks.math.columbia.edu/tag/00UO}{Tag 00UO}]{stacks-project}), 
even though \cref{t:differentials-cotangent-map} implies that $T_\a^*f_\infty$ is an isomorphism. 
\end{example}

In \cref{ex:not-formally-unramified}, the morphism $f_\infty$ 
actually fails to be locally of finite type at $\a$,
despite the fact that by \cref{t:finite-fiber} the 
fiber of $f_\infty$ through $\a$ and all nearby fibers are of finite type. 
This follows by observing that, if the morphism $f_\infty$ were locally 
of finite type at $\a$, then it would be unramified by 
\cite[\href{https://stacks.math.columbia.edu/tag/02FM}{Tag 02FM}]{stacks-project}, 
which we just saw is not the case. 
To better understand the failure of $f_\infty$ to be locally of finite type at $\a$, 
we revisit the example from this point of view. 

\begin{example}
\label{ex:not-quasi-finite}
Continuing with the same notation as in \cref{ex:not-formally-unramified}, 
let $R = \Supp \Om_{X/Y}$ be the ramification locus of $f$, which consists of the origin of $\A^1_k$. 
By looking at the equations \eqref{eq:not-formally-unramified} of $f_\infty$, 
it is clear that the only way to get a finitely generated extension out of this system is to localize at $x_0$, 
which corresponds to restricting $f_\infty$ to $(X \setminus R)_\infty$. Here we observe that, if $x_0$ is not inverted, its differential $dx_0$ is precisely the nonzero element
of $\Om_{X_\infty/Y_\infty}$ found in \cref{ex:not-formally-unramified}.
Note that the open set $(X \setminus R)_\infty$ is much smaller compared to $X_\infty \setminus R_\infty$;
for instance, it does not contain the point $\a \in X_\infty$ corresponding to the prime ideal $(x_0)$. 

Another way to see that $f_\infty$ is not locally of finite type at $\a$
relies on the following argument. 
It is a straightforward computation to check that both the residue field extension $k_\b \subset k_\a$ 
(where $\b = f_\infty(\a)$) and
the extension of function fields $k(Y_\infty) \subset k(X_\infty)$ are finite of degree $2$,
and it is clear that the fiber of $f$ over $\b$ consists only of $\a$ as a reduced point. 
If the local map $f_\infty^\sharp \colon \cO_{Y_\infty,\b} \to \cO_{X_\infty,\a}$ were essentially of finite type,  then it would follow by \cite[\href{https://stacks.math.columbia.edu/tag/052V}{Tag 052V}]{stacks-project} and the going-down theorem that the morphism $\Spec \cO_{X_\infty,\a} \to \Spec \cO_{Y_\infty,\b}$ is surjective. However, clearly the generic arc on $Y$ with order of contact 1 with the origin belongs to $\Spec \cO_{Y_\infty,\b}$ but not to the image of $\Spec \cO_{X_\infty,\a}$. 
\end{example}

The example discussed above suggests that given a morphism of finite type
$f \colon X \to Y$ with ramification locus $R$, 
the map on arc spaces $f_\infty \colon X_\infty \to Y_\infty$ can only be locally of finite type 
at points $\a \in (X \setminus R)_\infty$.
The next theorem establishes this fact in full generality, and provides
the precise link between the ramification locus of $f_\infty$ and that of $f$.

\begin{theorem}
\label{t:locally-finite-type}
Let $f \colon X \to Y$ be a morphism of finite type between schemes
over a perfect field $k$, and let $f_\infty \colon X_\infty \to Y_\infty$ be the induced morphism of arc spaces. 
For any $\a \in X_\infty$, the following are equivalent:
\begin{enumerate}
\item
\label{i1:locally-finite-type}
$f_\infty$ is unramified at $\a$;
\item
\label{i2:locally-finite-type}
$f_\infty$ is quasi-finite at $\a$;
\item
\label{i3:locally-finite-type}
$f_\infty$ is locally of finite type at $\a$;
\item
\label{i4:locally-finite-type}
$f$ is unramified at $\a(0)$.
\end{enumerate}
Moreover, the fiber of $f_\infty$ through $\a$ is locally of finite type at $\a$ if and only if $f$
is unramified at $\a(\e)$.
\end{theorem}

We deduce the following criterion for $f_\infty$ to be a morphism of finite type.

\begin{corollary}
\label{t:unramified-eq-finite-type}
With the same assumptions as in \cref{t:locally-finite-type}, the following are equivalent:
\begin{enumerate}
\item 
\label{i1:unramified-eq-finite-type}
$f_\infty$ is unramified;
\item 
\label{i2:unramified-eq-finite-type}
$f_\infty$ is quasi-finite;
\item 
\label{i3:unramified-eq-finite-type}
$f_\infty$ is of finite type;
\item 
\label{i4:unramified-eq-finite-type}
every fiber of $f_\infty$ is of finite type;
\item 
\label{i5:unramified-eq-finite-type}
$f$ is unramified.
\end{enumerate}
\end{corollary}

\begin{proof}
First note that as $f \colon X\to Y$ is quasi-compact, so is $f_\infty \colon X_\infty \to Y_\infty$. Indeed, there exists a cover by open affine sets $V$ of $Y$ such that the scheme-theoretic preimage $V \times_Y X$ is quasi-compact. By \cite[Proposition 2.3(2)]{NS10} we have $(V \times_Y X)_\infty \isom V_\infty \times_{Y_\infty} X_\infty$, and this is quasi-compact since $Z_\infty \to Z$ is affine for any $k$-scheme $Z$. As $V_\infty$ is an affine open subset of $Y_\infty$, we are done.

Therefore the implications \eqref{i1:unramified-eq-finite-type}$\Rightarrow$\eqref{i2:unramified-eq-finite-type}$\Rightarrow$\eqref{i3:unramified-eq-finite-type} follow from the definitions, and 
\eqref{i3:unramified-eq-finite-type}$\Rightarrow$\eqref{i4:unramified-eq-finite-type} 
from the fact that being of finite type is stable under base change. 
Finally, the implications \eqref{i4:unramified-eq-finite-type}$\Rightarrow$\eqref{i5:unramified-eq-finite-type}$\Rightarrow$\eqref{i1:unramified-eq-finite-type}
follow directly from \cref{t:locally-finite-type}.
\end{proof}

The proof of \cref{t:locally-finite-type} relies on the description 
of the sheaf of differentials of $X_\infty$ from \cite{dFD20}.  
Assuming for simplicity that $X = \Spec A$ is affine where $A$ is a $k$-algebra,
and writing $X_\infty = \Spec A_\infty$, the formula states that
\begin{equation}
\label{eq:dFD}
\Om_{X_\infty/k} \simeq \Om_{X/k} \otimes P_\infty
\end{equation}
where 
\begin{equation}
\label{eq:P_infty}
P_\infty := \frac{\LS{A_\infty}}{t\.A_\infty[[t]]}.
\end{equation}
We already used this formula in \cref{s:cotangent-map}. 
Adopting the same notation used there, for every arc $\a \in X_\infty$
we denote $P_\a := P_\infty \otimes k_\a$ and similarly if 
$B_\infty := A_\infty[[t]]$ then we set $B_\a := B_\infty \otimes \k_\a$.

\begin{lemma}
\label{l:MotimesPinfty}
With the above notation, let $M$ be a finitely generated $A$-module, 
and consider the $A_\infty$-module $M \otimes P_\infty$.
\begin{enumerate}
\item
\label{i1:MotimesPinfty}
We have $M\otimes P_\a \ne 0$ if and only if $M_{\a(\e)} \ne 0$, 
and if this occurs then $\dim_{k_\a}(M\otimes P_\a) = \infty$.
\item
\label{i2:MotimesPinfty}
We have $(M\otimes P_\infty)_\a \ne 0$ if and only if $M_{\a(0)} \ne 0$.
\end{enumerate}
\end{lemma}

\begin{proof}
Assume $M_{\a(0)} \ne 0$, the lemma being trivial otherwise. 

By the structure theorem over principal ideal domains, 
the $B_\a$-module $M \otimes B_\a$ is presented by a diagonal matrix with entries $t^{a_i}$ along the diagonal, 
and the rank of the free part of $M \otimes B_\a$ (which is the same as the number of 
zero rows in the presentation matrix) is equal to the rank of $M$ at $\a(\e)$. 
As $P_\a$ is $t$-divisible, 
tensoring by it kills the torsion part of $M \otimes B_\a$, 
hence we have $M \otimes P_\a \ne 0$ if and only $M_{\a(\e)} \ne 0$. 
Note furthermore that if this occurs then $M \otimes P_\a$ contains a direct summand isomorphic to $P_\a$, 
which is infinite dimensional over $k_\a$. 
This proves part \eqref{i1:MotimesPinfty}.

To prove \eqref{i2:MotimesPinfty}, we consider the $k_\a[[s]]$-valued arc $\f$ on $X$ given by
$\f(s,t) := \a(s+t)$. 
We can regard $\f$ as an arc on $X_\infty$, an infinitesimal deformation of $\a$.
Let $B_\f := B_\infty \otimes k_\a[[s]]$ and $P_\f := P_\infty \otimes k_\a[[s]]$.
Note that $M \otimes B_\f$ is now presented by a diagonal matrix 
with entries $(s+t)^{a_i}$ along the diagonal,
and tensoring with $P_\f$ no longer kills the cokernel since this module
is not $(s+t)$-divisible. 
This shows that if $M_{\a(0)} \ne 0$ then $(M\otimes P_\infty)_\a \ne 0$.
\end{proof}

\begin{proof}[Proof of \cref{t:locally-finite-type}]
The implications \eqref{i1:locally-finite-type}$\Rightarrow$\eqref{i2:locally-finite-type}$\Rightarrow$\eqref{i3:locally-finite-type} hold by definition. 
To prove that \eqref{i4:locally-finite-type} implies \eqref{i1:locally-finite-type}, assume that $f$ is unramified at $\a(0)$. By \cite[\href{https://stacks.math.columbia.edu/tag/0395}{Tag 0395}]{stacks-project} there exist an open neighborhood $U \subset X$ of $\a(0)$ and a factorization $U \to Z \to Y$ where $U \to Y$ is a closed immersion and $Z \to Y$ is \'etale. In particular, we have that $Z_\infty \isom Z \times_Y Y_\infty$ and thus $U_\infty \to Y_\infty$ is the composition of a closed immersion and the base change of an \'etale morphism, hence is unramified. As being unramified is a local property, it follows that $f_\infty \colon X_\infty \to Y_\infty$ is unramified at $\a$.

It remains to prove that \eqref{i3:locally-finite-type} implies \eqref{i4:locally-finite-type}.
We reduce to the case where $X = \Spec A$ is affine, and write $X_\infty = \Spec A_\infty$. 
Let $P_\infty$ as defined in \eqref{eq:P_infty}.
The first step is to observe that
\[
\Om_{X_\infty/Y_\infty} \simeq \Om_{X/Y} \otimes P_\infty,
\]
which follows from the formula in \eqref{eq:dFD} by tensoring the exact sequence
\[
\xymatrix@C=15pt{ \Om_{Y/k} \otimes \cO_X \ar[r] & \Om_{X/k} \ar[r] & \Om_{X/Y} \ar[r] & 0}       
\]
with $P_\infty$ and using that this operation is left exact. 

Assume that $f$ is ramified at $\a(0)$. Then $\Om_{X/Y,\a(0)} \ne 0$, 
hence 
\begin{equation}
\label{eq:ramified}
\Om_{X_\infty/Y_\infty,\a} \ne 0
\end{equation}
by part \eqref{i2:MotimesPinfty} of \cref{l:MotimesPinfty}.
We distinguish two cases.
If $f$ is unramified at $\a(\e)$, then 
\[
\Om_{X_\infty/Y_\infty} \otimes k_\a = \Om_{X/Y} \otimes P_\a = 0
\]
by part \eqref{i1:MotimesPinfty} of \cref{l:MotimesPinfty},
and contrasting this with \eqref{eq:ramified} shows that $f_\infty$ cannot be locally 
of finite type at $\a$ by 
\cite[\href{https://stacks.math.columbia.edu/tag/02FM}{Tag 02FM}]{stacks-project}.
If $f_\infty$ is ramified at $\a(\e)$, then
$\Om_{X_\infty/Y_\infty} \otimes k_\a$ is infinite dimensional over $k_\a$
(again by part \eqref{i1:MotimesPinfty} of \cref{l:MotimesPinfty}), 
and this implies that $f_\infty$ is not locally of finite type at $\a$. 

Note that the last part of the argument also implies that if $f$ is ramified at $\a(\e)$ then
the fiber of $f_\infty$ over $\b = f_\infty(\a)$ is not locally of finite type at $\a$, 
since in this case we have that
\[
\Om_{f^{-1}_\infty(\b)/k_\b} \otimes k_\a \simeq \Om_{X_\infty/Y_\infty} \otimes k_\a
\]
which is infinite dimensional over $k_\a$.
On the other hand, if $f_\infty$ is unramified at $\a(\e)$ then
we already proved in \cref{t:finite-fiber} that $f^{-1}(\b)$ is locally of finite type at $\a$,
so this finishes the proof of the theorem. 
\end{proof}

\begin{remark}
The proof of \cref{t:locally-finite-type} shows that the support of $\Om_{X_\infty/Y_\infty}$
is equal to the inverse image of the support of $\Om_{X/Y}$ under the truncation map $X_\infty \to X$. 
By \cref{t:locally-finite-type}, this implies that $\Supp \Om_{X_\infty/Y_\infty}$ is exactly the ramification locus of $f_\infty$, meaning that if $f_\infty$ fails to be unramified at a point $\a$, 
it fails both because it is not locally of finite type at $\a$ and because 
$\Om_{X_\infty/Y_\infty,\a} \ne 0$. 
\end{remark}

%
%

\section{Finiteness properties of local rings at stable points}

\label{s:stable-points}

We now turn to applications. In this section, we begin by looking at local rings at stable points, whose definition is recalled next.
We refer to \cite[Section~10]{dFD20} for a discussion of the equivalence of the following definition
with the original one given in \cite{Reg09} and for generalities on constructible subsets
of arbitrary schemes (see also \cite[\href{https://stacks.math.columbia.edu/tag/005G}{Tag 005G}]{stacks-project}).

\begin{definition}
Let $X$ be a scheme of finite type over a field $k$.
A point $\a \in X_\infty$ is said to be \emph{constructible}
if it is the generic point of a constructible subset of $X_\infty$.
Assuming that $X$ is a variety and $k$ is a perfect field,  
we say that a point $\a \in X_\infty$ is \emph{stable}
if it is constructible and $\a(\e) \in X_\sm$.
\end{definition}

Let henceforth $X$ be a variety over a perfect field $k$.
A systematic study of the local rings at stable points started in \cite{Reg06,Reg09}, where it
was proved among other things that if $\charK k = 0$ and $\a \in X_\infty$ is a stable point, 
then the completion $\^{\O_{X_\infty,\a}}$ of the local ring of $\a$ is Noetherian
(specifically, this property follow from \cite[Corollary~4.6]{Reg06}
and \cite[Theorem~3.13]{Reg09}, see also \cite[Theorem~4.1]{Reg06c}). 
This result was used in \cite{Reg06} to establish a curve selection lemma in the space of arcs, 
a result that plays a major role in all known proofs of the Nash problem, see \cites{LJR12,FdBPP12,dFD16}.
As it was observed later, the condition on the characteristic can be dropped, and the 
result extends to all varieties over a perfect field. 
A different proof was given in \cite{dFD20}.

Reguera's proof that the completion $\^{\O_{X_\infty,\a}}$ 
of the local ring at a stable point is Noetherian relies on showing that the maximal ideal of $\O_{(X_\infty)_\red,\a}$ is finitely generated and that the completion of the reduced local ring agrees with that of $\O_{X_\infty,\a}$. 
A different proof of the result was later given in \cite[Corollary~10.13]{dFD20}.
These results led Reguera to raise the following question.

\begin{question}[{\cite[Question~4.8]{Reg09}}]
\label{q:Noetherian}
If $\a \in X_\infty$ is a stable point, 
is the reduced local ring $\O_{(X_\infty)_\red,\a}$ Noetherian?
\end{question}

It should be noted that restricting the question to the reduced local ring is a necessary step, 
since in general the local ring $\O_{X_\infty,\a}$ at a stable point can fail to be Noetherian. 
This was already observed in a concrete example in \cite[Example~3.16]{Reg09}
by looking at the nilpotent elements of the local ring when $X$ is a cuspidal curve 
and $\a$ is the generic arc giving the parameterization of the cusp; see also \cref{eg:node,eg:cusp}. 
In fact, it is a general fact that local rings of $X_\infty$
are never Noetherian if they are non-reduced. 
In the following proposition we restrict the attention to stable points
since local rings at non-stable points are clearly not Noetherian.

\begin{proposition}
\label{p:not-reduced-not-noetherian}
For any stable point $\a \in X_\infty$, if the local ring 
$\O_{X_\infty,\a}$ is not reduced then it is not Noetherian. 
\end{proposition}

\begin{proof}
By \cite[Theorem~3.13]{Reg09} (when $\charK k = 0$) and \cite{BH21} (in general), 
the nilradical of $\cO_{X_\infty,\a}$ is contained in $\bigcap_n \fm_\a^n$.
If $\cO_{X_\infty,\a}$ is not reduced, then its
$\fm_\a$-adic topology is not separated and therefore, by Krull's intersection theorem, the ring cannot be Noetherian.
\end{proof}

Using the results of arc fibers from previous sections, we can prove that
if $\a$ is a stable point then $\Spec \cO_{X_\infty,\a}$ is Noetherian as a topological space
and the maximal ideal of $\O_{X_\infty,\a}$ is finitely generated. 
The first property can be viewed as an intermediate step toward a positive answer to \cref{q:Noetherian}. 
The second property is claimed in \cite[Theorem 4.1]{Reg06}, but the proof given there requires 
restricting to $(X_\infty)_\red$, so the result only applies to 
$\O_{(X_\infty)_\red,\a}$ (cf.\ \cite{Reg06c}). 
While it is quite possible that Reguera's proof could be adjusted so that one can
avoid having to restrict to the reduced subscheme, 
we propose here an alternative approach which provides a quick proof 
based on general linear projections and our previous results on jet fibers. 
This is quite different from Reguera's approach, 
which involves looking at the truncation maps to the jet schemes. 

For convenience, we introduce the following terminology. 

\begin{definition}
A scheme is said to be \emph{topologically Noetherian} if its underlying topological space is Noetherian.
\end{definition}

In the following, we denote by $\dim(A)$ the Krull dimension
of a local ring $A$ and by $\embdim(A)$ its embedding dimension. 

\begin{theorem}
\label{t:fg-top-noetherian}
Let $X$ be a variety over a perfect field $k$ and $\a \in X_\infty$ a stable point.
\begin{enumerate}
\item
\label{i1:fg-top-noetherian}
The maximal ideal of the local ring $\cO_{X_\infty,\a}$ is finitely generated.
\item
\label{i2:fg-top-noetherian}
The scheme $\Spec \cO_{X_\infty,\a}$ is topologically Noetherian and 
\[
\dim (\cO_{X_\infty,\a}) \leq \embdim (\cO_{X_\infty,\a}).
\]
\end{enumerate}
\end{theorem}

\begin{proof}
We may assume $X \subset \A^n_k$ is affine. Let $d = \dim X$. 
Since $k$ is perfect the extension $k \subset k(X)$ is separable and hence there exists a transcendence basis $x_1,\ldots,x_d$ for $k(X)$ such that $k(x_1,\ldots,x_d) \subset k(X)$ is finite separable. We may assume that $x_1,\ldots,x_d \in k[X]$ and thus they define a generically finite morphism $f \colon X \to Y = \A^d_k$ that is generically \'etale by \cite[\href{https://stacks.math.columbia.edu/tag/090W}{Tag 090W}]{stacks-project}. 

Now let $\a \in X_\infty$ be a stable point and set 
$\b = f_\infty(\a)$, which is again stable (e.g., see \cref{c:reguera-4-8}). 
Since $Y$ is smooth, the local ring $\cO_{Y_\infty,\b}$ is 
the localization of a polynomial algebra over a regular ring at 
a prime of finite height, hence it is
regular and its maximal ideal $\fm_\b$ is finitely generated. Using \cref{t:unramified-reduced-arc-fiber} we see that $\fm_\a = \fm_\b \cdot \cO_{X_\infty,\a}$
and this proves \eqref{i1:fg-top-noetherian}.

We now address \eqref{i2:fg-top-noetherian}.
By \cref{c:cotangent-map-isom}, the cotangent map
\[
T^*_\a f_\infty \colon \fm_\b/\fm_\b^2 \otimes_{k_\b} k_\a \to \fm_\a/\fm_\a^2
\]
is an isomorphism, and hence 
\[
\embdim(\cO_{Y_\infty,\b}) = \embdim(\cO_{X_\infty,\a}).
\]
Furthermore, as we may assume that $f$ is quasi-finite,  
\cref{t:finite-fiber} implies that the induced map 
$\Spec \O_{X_\infty,\a} \to \Spec \O_{Y_\infty,\b}$ has topologically finite fibers. 

Since $\cO_{Y_\infty,\b}$ is a regular ring of dimension
$\dim (\cO_{Y_\infty,\b}) = \embdim(\cO_{Y_\infty,\b})$,
\cref{l:top-noetherian} implies that 
$\Spec \cO_{X_\infty,\a}$ is topologically Noetherian and 
\[
\dim(\cO_{X_\infty,\a}) \le \embdim(\cO_{Y_\infty,\b}).
\] 
The stated bound on the dimension follows by combining the last two displayed formulas. 
\end{proof}

\begin{lemma}
\label{l:top-noetherian}
Let $f \colon X \to Y$ be a morphism of schemes, and assume that $Y$ is topologically Noetherian.
If $f$ has topologically finite fibers, then $X$ is topologically Noetherian and $\dim X \leq \dim Y$.
\end{lemma}

\begin{proof}
Let us first show that $\dim X \leq \dim Y$. To that end, let $d = \dim X$ and let $x_0,\ldots,x_d \in X$ be distinct points with $x_i$ a specialization of $x_{i-1}$ for all $i=1,\ldots,d$. Then $f(x_i)$ is a specialization of $f(x_{i-1})$, and since each point $x_i$ is isolated in its fiber, the points $f(x_0),\ldots,f(x_d)$ are all distinct,  showing that $\dim Y \ge d$. 

To show that $X$ is topologically Noetherian, it remains to prove that each closed subset $Z \subset X$ has only finitely many irreducible components. Assume on the contrary that there exists a closed subset $Z\subset X$ 
with infinitely many irreducible components $Z_i$, $i\in I$. 
For each $i$ let $z_i \in Z_i$ denote the generic point. As $f$ has topologically finite fibers, there exists an infinite subset $I' \subset I$ such that the points $f(z_i)$, for $i\in I'$, are pairwise distinct. 
For every $i \in I'$, let $W_i = \overline{\{f(z_i)\}}$. As $\dim Y < \infty$, all possible chains between these sets $W_i$ have finite length and thus there exists an infinite subset $I'' \subset I'$ such that $W_i \not\subset W_j$ for every $i,j \in I''$ with $i\neq j$. Then $I''$ parameterizes
infinitely many irreducible components $W_i$ of $\overline{f(Z)}$,
which is impossible if $Y$ is topologically Noetherian.
\end{proof}

\begin{remark}
\cref{t:fg-top-noetherian} gives a new proof of the fact, 
originally due to \cite[Proposition~3.7(iv)]{Reg09},  that 
local rings at stable points have finite Krull dimension.
\end{remark}

%
%

\section{On embedding dimension and codimension of arc spaces}

\label{s:semicontinuity-embdim}

Given a variety $X$ over a perfect field $k$, 
the embedding dimension of the local ring of the arc space $X_\infty$
at a point $\a$ was used in \cite{dFD20} as a way of measuring the `codimension'
of $\a$ in $X_\infty$. For a precise statement see \cite[Theorems~C]{dFD20},
where the embedding dimension at $\a$ was shown to equal the jet codimension of $\a$,
which in turn is computed from the truncations of $\a$ in the jet schemes of $X$.
The following characterization of stable points 
in terms of the embedding dimension of their local rings
was also established there. 

\begin{theorem}[{\cite[Theorem E]{dFD20}}]
\label{t:stable-points-finite-embdim}
Let $X$ be a variety over a perfect field $k$. For any $\a \in X_\infty$, we have
$\embdim (\cO_{X_\infty,\a}) < \infty$ if and only if $\a$ is a stable point.
\end{theorem}

\begin{remark}
\label{r:stable-points-finite-embdim}
More generally, if $X$ is any scheme of finite type over a perfect field $k$, 
the condition that the embedding dimension is finite characterizes 
constructible points $\a \in X_\infty$ whose generic points $\a(\e)$
are contained in the smooth locus of $X$ (see \cite[Theorem~10.8]{dFD20}). 
\end{remark}

An application of the results of \cref{s:cotangent-map} gives us the next theorem. 
One should compare the theorem and its corollary to \cite[Theorem~D]{dFD20}
and \cite[Propositions~4.1 and~4.5]{Reg09}.

\begin{theorem}
\label{c:embdim-smooth-unramified}
Let $f \colon X \to Y$ be a morphism of schemes over a perfect field $k$.
Let $\a \in X_\infty$ and $\b = f_\infty(\a)$.
\begin{enumerate}
\item
\label{i1:embdim-smooth-unramified}
If $f$ is unramified at $\a(\e)$, then 
\[
    \embdim(\O_{X_\infty,\a}) \le \embdim(\O_{Y_\infty,\b}).
\]
\item
\label{i2:embdim-smooth-unramified}
If $X$ and $Y$ are locally of finite type over $k$ and $f$ is smooth at $\a(\e)$, then 
\[
    \embdim(\O_{Y_\infty,\b})  \le \embdim(\O_{X_\infty,\a}) + \ord_\a (\Fitt^r \Om_{X/Y}),
\]
where $r = \dim(\Om_{X/Y} \otimes k_{\a(\e)})$, 
and equality holds if $X$ is smooth at $\a(0)$ and $f$ is \'etale at $\a(\e)$.
\end{enumerate}
\end{theorem}

\begin{proof}
It follows directly from \cref{t:differentials-cotangent-map}. 
\end{proof}

\begin{remark}
Just like in \cref{r:differentials-cotangent-map}, 
in the setting of \eqref{i2:embdim-smooth-unramified} of \cref{c:embdim-smooth-unramified}
we also have the bound
\[
    \embdim(\O_{Y_\infty,\b})  \le \embdim(\O_{X_\infty,\a}) + \ord_\a(\Fitt^r \Om_{X/Y}) 
    - \ord_\a(\Fitt^d \Om_{X/k}) + \ord_\a(\Fitt^e \Om_{Y/k}).
\]
where $d = \dim (\Om_{X/k} \otimes k_{\a(\e)})$ and $e = \dim (\Om_{Y/k} \otimes k_{\b(\e)})$.
\end{remark}

As an immediate consequence of \cref{c:embdim-smooth-unramified}, 
we obtain the following property. 
Part \eqref{i1:reguera-4-8} was proved in \cite[Proposition~4.5(i)]{Reg09}
assuming $\charK k = 0$, and part \eqref{i2:reguera-4-8} answers \cite[Question 4.8]{Reg09}.

\begin{corollary}
\label{c:reguera-4-8}
Let $f \colon X \to Y$ be a dominant morphism of varieties over a perfect field $k$.
Let $\a \in X_\infty$ and $\b = f_\infty(\a) \in Y_\infty$. 
\begin{enumerate}
\item
\label{i1:reguera-4-8}
If $f$ is generically finite and separable,
then $\a$ is stable if and only if $\b$ is stable.
\item
\label{i2:reguera-4-8}
If $f$ is generically smooth and $\a$ is stable, then $\b$ is stable.
\end{enumerate}
\end{corollary}

\begin{proof}
It follows from \cref{t:stable-points-finite-embdim} and \cref{c:embdim-smooth-unramified}.
\end{proof}

\begin{remark}
More generally, we see by \cref{r:stable-points-finite-embdim} that 
if $f \colon X \to Y$ is a dominant morphism of schemes that is locally of finite type over a field $k$
and $\a \in X_\infty$ is a constructible point
such that both $X$ and $f$ are smooth at $\a(\e)$, then $\b = f_\infty(\a)$
is a constructible point of $Y_\infty$.
\end{remark}

In \cite{CdFD}, we look at the embedding codimension of the local rings of $X_\infty$ as a way
of measuring its singularities. 
Since we are dealing with non-Noetherian rings, where the notion of embedding codimension
was only introduced recently, we recall the definition from \cite{CdFD}. 

\begin{definition}
The \emph{embedding codimension} of a local ring $(A,\fm,k)$ is defined by
\[
\embcodim(A) := \height(\ker(\g))
\]
where $\g \colon \Sym_k(\fm/\fm^2) \to \gr(A) = \bigoplus_{n \ge 0} \fm^n/\fm^{n+1}$ is the natural homomorphism. 
\end{definition}

If $A$ is Noetherian, then $\embcodim(A) = \embdim(A) - \dim(A)$, which justifies the terminology, 
but the definition clearly needed to be modified beyond the Noetherian case
since the right-hand-side of this equation may be negative or even 
be a difference of infinities. 
The interpretation of the above definition is that as the embedding dimension is the
dimension of the tangent space, the embedding codimension is the codimension 
of the tangent cone in the tangent space.
For a more in-depth discussion, see \cite[Section 6]{CdFD}.

Just like the embedding dimension is used to characterize stable points, 
the embedding codimension characterizes arcs that are not fully contained in the singular locus. 
The following result is established in \cite{CdFD} under the 
assumption that either $\charK k =0$ or $\a \in X_\infty(k)$.\footnote{This assumption
was omitted in part~(2) of \cite[Theorem~B]{CdFD}; the result is correctly stated in \cite[Theorem~8.5]{CdFD}.}
 
\begin{theorem}
\label{t:embcodim-bound}
Let $X$ be a scheme of finite type over a perfect field $k$. For any $\a \in X_\infty$, 
we have $\embcodim(\O_{X_\infty,\a}) < \infty$ if and only if $X$ is smooth at $\a(\e)$.
Furthermore, if this occurs and $X^0 \subset X$ is the irreducible component containing $\a(\e)$, then
\[
\embcodim(\O_{X_\infty,\a}) \le \ord_\a(\Jac_{X^0})
\]
where $\Jac_{X^0}$ is the Jacobian ideal of $X^0$. 
\end{theorem} 

\begin{proof}
This is proven in the same way as \cite[Theorem~8.5 and Corollary~8.8]{CdFD}, with \cref{c:cotangent-map-isom} substituting \cite[Theorem~8.1]{CdFD}. 
\end{proof}

%
%

\section{Semicontinuity and constructibility properties}

\label{s:semicont}

Embedding dimension and embedding codimension of Noetherian local rings were studied 
by Lech in \cite{Lec64}. One of the main results proved there (see \cite[Theorem~3]{Lec64}) is that if 
$\fp$ is a prime ideal of a Noetherian local ring $(A,\fm)$, then  
\[
\embdim(A) \ge \dim(A/\fp) + \embdim(A_\fp).
\]
He deduces from this formula that if
$\dim(A) = \dim(A/\fp) + \dim(A_\fp)$, then 
\[
\embcodim(A) \ge \embcodim(A_\fp).
\]
Such results are important because they show that these invariants satisfy expected
semicontinuity properties. This becomes particularly relevant in the moment
we want to consider the embedding codimension as a measure of singularities. 
Lech's proof of these results uses in an essential way the assumption that the local ring is Noetherian, 
and does not extend beyond this case. 

Here we apply the results of the previous sections to prove that 
the embedding dimension and embedding codimension of arc spaces
satisfy similar semicontinuity properties. 
It may be worthwhile to note that the proof of the following theorem does not depend on \cite{Lec64}, 
but rather uses the theorems on arc fibers to reduce the problem 
to polynomial rings where embedding dimensions and embedding codimensions are 
easily computable.

\begin{theorem}
\label{t:semicont}
Let $X$ be a scheme locally of finite type over a perfect field $k$. 
Let $\a,\a' \in X_\infty$ be two points with $\a'$ specializing to $\a$, 
and let $c = \codim(\a,\a')$.
\begin{enumerate}
\item
\label{i1:semicont}
We have
\[
\embdim(\O_{X_\infty,\a}) \ge c + \embdim(\O_{X_\infty,\a'}).
\]
\item
\label{i2:semicont}
If $\embdim(\O_{X_\infty,\a}) < \infty$
and $\dim(\gr(\O_{X_\infty,\a})) \le c + \dim(\gr(\O_{X_\infty,\a'}))$, 
then 
\[
\embcodim(\O_{X_\infty,\a}) \ge \embcodim(\O_{X_\infty,\a'}).
\]
\end{enumerate}
\end{theorem}

\begin{proof}
Clearly \eqref{i1:semicont} holds if $\embdim(\O_{X_\infty,\a}) = \infty$, so
we can assume that $\embdim(\O_{X_\infty,\a}) < \infty$.
\cref{r:stable-points-finite-embdim} implies that
$\a(\e)$ is contained in the smooth locus of $X$. 
Therefore there is an irreducible component $X_0$ and an open set $U \subset X_0$ such that 
$\a(\e) \in U$ and $X$ is smooth along $U$, which in particular implies that 
$U$ does not intersect any other irreducible component of $X$. 
Note that for every point $\g \in \Spec \O_{X_\infty,\a}$ 
we have $\g(\e) \in \Spec \O_{X,\a(\e)} = \Spec \O_{U,\a(\e)}$, hence $\g(\e) \in U$. 

We can assume without loss of generality that $X \subset \A^n_k$ is an affine scheme. 
Denoting $d = \dim X_0$, it follows by 
\cref{c:cotangent-map-isom} and \cref{r:cotangent-map-isom}
that if $f \colon X \to \A^d_k$ is a general linear projection  
then the induced maps on tangent spaces $T_\a f_\infty$ and $T_{\a'} f_\infty$ are both isomorphisms.
Setting $\b = f_\infty(\a)$ and $\b' = f_\infty(\a')$, this means that
\begin{equation}
\label{eq1:semicont}
\embdim(\O_{X_\infty,\a}) = \embdim(\O_{Y_\infty,\b}) 
\quad\text{and}\quad
\embdim(\O_{X_\infty,\a'}) = \embdim(\O_{Y_\infty,\b'}).
\end{equation}
Additionally, we may assume that $f$ induces a finite map $X_0 \to \A^d_k$. 
In particular, $f$ is quasi-finite at every point of $U$, 
hence at every point of the form $\g(\e) \in X$ where $\g \in \Spec \O_{X_\infty,\a}$.
We can therefore apply part \eqref{i1:finite-fiber-1} of \cref{t:finite-fiber-1}, 
which implies that $f_\infty$ restricts to a topologically finite map
$\Spec \O_{X_\infty,\a} \to \Spec \O_{Y_\infty,\b}$. 

Denoting $\fp_{\a'} \subset \O_{X_\infty,\a}$ the ideal of $\a'$, let 
\[
\fp_{\a'} = \fp_0 \subsetneq \fp_1 \subsetneq \dots \subsetneq \fp_l = \fm_\a \subset \O_{X_\infty,\a}
\]
be a chain of prime ideal of $\O_{X_\infty,\a}$ computing the height of $\fp_{\a'}$. 
Setting $\fq_i = \fp_i \cap \O_{Y_\infty,\b}$, we obtain as chain of prime ideals
\[
\fq_{\b'} = \fq_0 \subsetneq \fq_1 \subsetneq \dots \subsetneq \fq_l = \fm_\b \subset \O_{Y_\infty,\b},
\]
the inclusions remaining strict because
the map $\Spec \O_{X_\infty,\a} \to \Spec \O_{Y_\infty,\b}$ has topologically finite fibers. 
Here we are denoting by $\fq_{\b'}$ the ideal of $\b'$ in $\O_{Y_\infty,\b}$. 

Now, by \eqref{eq1:semicont} we have $\embdim(\O_{Y_\infty,\b}) < \infty$, hence
\cref{t:stable-points-finite-embdim} implies that $\b$ is a stable point of $Y_\infty$, which is the arc space of 
an affine space. It follows that $\O_{Y_\infty,\b}$ is isomorphic to the localization at the maximal ideal
of a finitely generated polynomial ring over a field. 
In particular, this ring is catenary, hence we have
\begin{equation}
\label{eq2:semicont}
\dim(\O_{Y_\infty,\b}/\fp_{\b'}) \ge l = \codim(\a,\a').
\end{equation}
Furthermore, for such a ring we have 
\[
\embdim(\O_{Y_\infty,\b}) = \codim(\b,\b') + \embdim(\O_{Y_\infty,\b'}),
\]
and \eqref{i1:semicont} follows by combining this formula with
\eqref{eq1:semicont} and \eqref{eq2:semicont}.

Turning now to \eqref{i2:semicont}, by \cite[Proposition~6.2]{CdFD} we have
\begin{equation}
\label{eq:embdim-formula}
\embdim(\O_{X_\infty,\a}) = \dim(\gr(\O_{X_\infty,\a})) + \embcodim(\O_{X_\infty,\a}).
\end{equation}
As we are assuming that $\embdim(\O_{X_\infty,\a}) < \infty$, we can 
solve for $\embcodim(\O_{X_\infty,\a})$. 
Similarly, \eqref{i1:semicont} implies that $\embdim(\O_{X_\infty,\a'}) < \infty$ as well,
hence the same argument can be repeated for $\O_{X_\infty,\a'}$
and \eqref{i2:semicont} follows from \eqref{i1:semicont}.
\end{proof}

\begin{remark}
When $X$ is a variety, \cref{t:semicont} can alternatively be proved by using the 
interpretation of embedding dimension as jet-codimension from \cite{dFD20}, which makes use of
the structure of the truncation maps to the jet schemes. 
Each proof gives a different perspective on semicontinuity and
provides a unique set of tools that can be use to further investigate this property.
\end{remark}

\begin{remark}
In view of the characterization of stable points given in \cref{t:stable-points-finite-embdim}, 
\cref{t:semicont} yields a new proof of the fact that 
every generalization of a stable point is itself a stable point, a property
first established in \cite[Proposition~3.7(vi)]{Reg09}.
\end{remark}

The condition that $\dim(\gr(\O_{X_\infty,\a})) \le \codim(\a,\a') + \dim(\gr(\O_{X_\infty,\a'}))$
imposed in part \eqref{i2:semicont} of \cref{t:semicont} can be viewed as a generalization of
the condition that $\dim(A) = \dim(A/\fp) + \dim(A_\fp)$ appearing in Lech's theorem, 
which guarantees that $\fp$ belongs to a chain of primes in $A$ computing the dimension;
they are clearly the same condition if $\O_{X_\infty,\a}$ is Noetherian. 
It is not clear how the condition should be formulated if $\embdim(\O_{X_\infty,\a}) = \infty$, 
which prompts the question whether anything can be said about semicontinuity of 
the embedding codimension at arcs of infinite embedding dimension. 

We want to present here a complementary result in this direction which sheds some light on
the behavior of the embedding codimension at the $k$-rational points of $X_\infty$.
The proof uses the existence of finite-dimensional formal models as established in the Drinfeld--Grinberg--Kazhdan theorem \cite{GK00,Dri02} in combination with an explicit formula for the embedding dimension of Drinfeld models obtained in \cite{CdFD}. We first recall the following definition \cite[D\'efinition~(9.3.1)]{ega-iii-pt1}.

\begin{definition}
A function $h$ from a topological space $X$ to a set $T$ is \emph{constructible}
if it has finite image and $h^{-1}(t)$ is constructible for every $t \in T$. 
The function $h$ is \emph{locally constructible} if every point $x \in X$ admits
an open neighborhood $V \subset X$ such that $h|_V$ is constructible. 
\end{definition}

%

\begin{theorem}
\label{t:embcodim-constr}
Let $X$ be a scheme of finite over a perfect field $k$, and let
$X_\infty^\o := X_\infty \setminus (\Sing X)_\infty$.
For every $\a \in X_\infty^\o$ denote by $X^\a$ the unique irreducible 
component of $X$ containing $\a(\e)$, and for every integer $e \ge 0$ define
\[
X_\infty^e := \{ \a \in X_\infty^\o \mid \ord_\a(\Jac_{X^\a}) = e\}.
\]
Then the function
\[
X_\infty^\o(k) \to \Z, 
\quad \a \mapsto \embcodim(\O_{X_\infty,\a})
\]
is locally constructible and restricts to a lower-semicontinuous function 
on $X_\infty^e(k)$ for every $e$.
\end{theorem}

\begin{proof}
First note that
\[
X_\infty^\o = \bigsqcup_{e=0}^\infty X_\infty^e = \{\a \in X_\infty \mid \embcodim(\O_{X_\infty,\a}) < \infty\},
\]
the second equality following by \cref{t:embcodim-bound};
in particular, the embedding codimension defines a function on this set with values in $\Z$.
Note that each set $X_\infty^e$ is constructible in $X_\infty^\o$. 
Therefore it suffices to prove that the function is lower-semicontinuous on 
$X_\infty^e(k)$ for every $e$, as this implies that the function is constructible on the open sets
$\bigsqcup_{e = 0}^r X_\infty^e(k) \subset X_\infty^\o(k)$ and hence locally constructible on $X_\infty^\o(k)$. 

We may assume that $X \subset \A^n_k$ is affine.
Fix $e \ge 0$, let $\a \in X_\infty^e(k)$ be any point, and set $d = \dim X^\a$ and $c = n-d$.
Proceeding as in \cite[Section~10]{CdFD}, 
let $X' \supset X^\a$ be the complete intersection scheme 
defined by the vanishing of $c$ general linear combinations $p_1,\dots,p_c$ of 
a set of generators of the ideal of $X^\a$ in $\A^n$. 
As discussed in \cite{CdFD}, we have
$\widehat{\O_{X_\infty,\a}} \simeq \widehat{\O_{X^\a_\infty,\a}} \simeq \widehat{\O_{X'_\infty,\a}}$.
Pick general affine coordinates $x_1,\dots,x_d,y_1,\dots,y_c$ in $\A^n_k$ and
let $\Delta = \det \left(\frac{\partial p_i}{\partial y_j}\right)_{i,j\leq m}$.

With the above choices, we have
\[
\ord_\a(\D) = \ord_\a(\Jac_{X'}) = \ord_\a(\Jac_{X^\a}) = e.
\]
In fact, the same equations hold if we replace $\a$ by any other $k$-rational point 
in some open neighborhood $X_\infty^{\D,e}$ of $\a$ in $X_\infty^e$, without changing the
choice of coordinates and of the linear combinations of 
the set of generators of the ideal of $X^\a$ made above. We can therefore assume without loss of 
generality that $\a$ is an arbitrary $k$-rational point of $X_\infty^{\D,e}$.
Since the open sets $X_\infty^{\D,e}$ constructed in this way cover $X_\infty^e$, 
it suffices to prove that the embedding codimension function 
is lower-semicontinuous on $X_\infty^{\D,e}(k)$. 

Following the construction outlined in \cite[Section~10]{CdFD}, let $Z \subset \A^{e(1 + 2d + c)}_k$
be the scheme defined by the equations
listed in \cite[Equation~(10a)]{CdFD}. 
Consider the morphism of schemes $w \colon X_\infty^{\Delta,e} \to Z$ defined as follows: 
if $\g$ is an $R$-valued point of $X_\infty^{\Delta,e}$ given by $(x(t),y(t)) \in R[[t]]^d \times R[[t]]^c$, 
then $w(\g)$ is given by the tuple $(t^e,\bar{x}(t),\bar{y}(t))$ where 
$\bar{x}(t) \equiv x(t) \mod t^{2e}$ and $\bar{y}(t) \equiv y(t) \mod t^e$.
By the theorem of Drinfeld, Grinberg and Kazhdan \cite{GK00,Dri02}, for each $\a \in X_\infty^{\Delta,e}(k)$ 
we have an isomorphism of completions
\[
   \^{\cO_{X_\infty,\a}} \isom \^{\cO_{Z,w(\a)}} \cotimes k[[t_i \mid i\in \N]],
\]
and by \cite[Proposition~7.6]{CdFD} we have that $\embcodim(\cO_{X_\infty,\a}) = \embcodim(\cO_{Z,w(\a)})$. 

Let $f \colon X \to Y := \A^d$ be the morphism induced by the projection $(x,y) \mapsto x$.
Let $\b = f_\infty(\a)$ and write $\b_n \in Y_n$ for the image of $\b$ in $Y_n$. 
By \cite[Theorem~10.2(2)]{CdFD}, for each $\a \in X_\infty^{\Delta,e}(k)$ the composition
\[
\^{\cO_{Z,w(\a)}} \hookrightarrow \^{\cO_{X_\infty,\a}} \to \^{\cO_{Y_\infty,\b}} \to \^{\cO_{Y_{2e-1},\b_{2e-1}}}
\]
gives an efficient formal embedding, that is, a closed embedding whose induced map on cotangent spaces is an isomorphism. 
In particular, $\embdim(\cO_{Z,w(\a)}) = \embdim (\cO_{Y_{2e-1},\b_{2e-1}})$. 
As $Y$ is smooth and $\b_{2e-1}$ is a $k$-rational point, it follows that the function 
\[
X_\infty^{\Delta,e}(k) \to \Z, 
\quad
\a \mapsto \embdim(\cO_{Z,w(\a)})
\]
is constant. On the other hand, since $Z$ is of finite type over $k$, the function 
\[
X_\infty^{\Delta,e}(k) \to \Z, 
\quad
\a \mapsto \dim_{w(\a)} Z
\]
is upper-semicontinuous, as it is the composition of a continuous map and an upper-semicontinuous map. Since $\embcodim(\cO_{Z,z}) = \embdim(\cO_{Z,z}) - \dim_z Z$ for all $z\in Z$, we get the assertion.
\end{proof}

\section{Divisorial valuations and log discrepancies}

\label{s:div-val}

Let $X$ be a variety over a perfect field $k$. 
In this paper, we adopt the following definition of divisorial valuation.

\begin{definition}
\label{d:div-val}
A \emph{divisorial valuation} on $X$ is a valuation $v$ of the function field
of $X$ with values in $\Z$ and center in $X$, whose residue field $k_v$
has transcendence degree over $k$ equal to $d-1$ where $d = \dim X$. 
\end{definition}

We stress that we are not requiring that $\Z$ is the value group of $v$; for instance, 
we distinguish a valuation $v$ from $2v$. 
According to the definition, a divisorial valuation is a valuation of the form
$v = q\ord_E$ where $q \in \N$ and $E$ is a prime divisor
on a normal variety $Y$ equipped with a proper birational morphism $f\colon Y \to X$
(cf.\ \cite[Lemma~2.45]{KM98}).

\begin{definition}
For any divisorial valuation $v=q\ord_E$ as above, we define the \emph{Mather log discrepancy}
\[
\^a_v(X) := q(\ord_E(\Jac_f) + 1)
\]
and the \emph{Mather--Jacobian log discrepancy}
\[
a_v^{\MJ}(X) := q(\ord_E(\Jac_f) - \ord_E(\Jac_X) + 1).
\]
\end{definition}

The definitions of Mather log discrepancy and Mather--Jacobian log discrepancy are clearly independent of the choice of the model $f \colon Y \to X$. 
These invariants of singularities where studied for instance in 
\cite{dFEI08,Ish13,dFD14,EI15}.

Next, we recall the definitions of divisorial arc and maximal divisorial
arc, following \cite{dFD20}. These notions trace back at least to \cite{ELM04}. 

\begin{definition}
A point $\a \in X_\infty$ is a \emph{divisorial arc}
if $\ord_\a$ extends to a divisorial valuation $v_\a$ of the function field of $X$.
A divisorial arc $\a$ is a \emph{maximal divisorial arc} if
$\a$ is maximal (with respect to specialization)
among all divisorial arcs $\g \in X_\infty$ with $v_\g = v_\a$.
\end{definition}

We recall the following property about maximal divisorial arcs. 

\begin{theorem}[{\cite[Theorem~11.4]{dFD20}}]
\label{t:edim=Mather}
Let $X$ be a variety over a perfect field $k$.
For every divisorial valuation $v$ on $X$ there exists a unique 
maximal divisorial arc $\a_v \in X_\infty$ whose associated valuation is equal to $v$, and
\[
\embdim(\O_{X_\infty,\a_v}) = \^a_v(X).
\]
\end{theorem}

In concrete terms, if $v = q\ord_E$ and $f \colon Y \to X$ is as in \cref{d:div-val}, then
the associated maximal divisorial arc 
$\a_v$ is the image of the generic arc on $Y$ with order of contact $q$ with $E$. 

By \cref{t:stable-points-finite-embdim,t:edim=Mather}
every maximal divisorial arc is a stable point of $X_\infty$. 
Conversely, the next theorem proves that every stable point whose special point is not the generic point of $X$
is a divisorial arc. 
When $\charK k = 0$, this property follows from the results of \cite{ELM04} (when $X$ is smooth)
and \cite{dFEI08} (in general). 
We give here a new, characteristic-free proof and use the property to 
establish lower-bounds on the embedding dimension at a stable point
and the dimension of the completed local ring.

\begin{theorem}
\label{t:stable=divisorial+bounds}
Let $X$ be a variety over a perfect field $k$ and $\a \in X_\infty$ a stable point. 
Assume that $\a(0)$ is not the generic point of $X$. 
\begin{enumerate}
\item
\label{i1:stable=divisorial+bounds}
The valuation $v = v_\a$ defined by $\a$ is divisorial.
\item
\label{i2:stable=divisorial+bounds}
If $\a_{v} \in X_\infty$ is the maximal divisorial arc associated to $v$
and $c = \codim(\a,\a_v)$, then
\begin{equation}
\label{eq:bound-Mather}
\embdim(\O_{X_\infty,\a}) = c + \embdim(\O_{X_\infty,\a_v}) = c + \^a_{v}(X)
\end{equation}
and
\begin{equation}
\label{eq:bound-MJ}
\dim (\^{\cO_{X_\infty,\a}}) \geq c + a_{v}^{\MJ}(X).
\end{equation}
Moreover, $\codim(\a,\a_v)$ is birationally invariant, in the sense that
for every proper birational morphism $X' \to X$, if
$\a'$ and $\a_v'$ are the lifts of $\a$ and $\a_v$ to $X'_\infty$ then
$\codim(\a',\a_v') = \codim(\a,\a_v)$. 
\end{enumerate}
\end{theorem}

\begin{proof}
We may assume that $X \subset \A^n_k$ is affine. Let 
$f \colon X \to Y = \A^d_k$ be a general linear projection, where $d = \dim X$.
Every stable point $\a \in X_\infty$ defines a valuation $v_\a$
of the function field of $X$, and its image
$\b = f_\infty(\a) \in Y_\infty$ is stable by \cref{c:reguera-4-8}.
Note that since $\a(0)$ is not the generic point of $X$, the valuation $v_\a$ is 
nontrivial; similarly, $\b(0)$ is not the generic point of $Y$, hence 
the valuation $v_\b$ defined by $\b$ is nontrivial. 

We claim that $\b$ is a divisorial arc. 
Note that this implies that $\a$ is also divisorial. Indeed, 
by construction $v_\b$ is the restriction of $v_\a$ to the function field of $Y$, hence
$k_{v_\b} \subset k_{v_\a}$. 
Since divisorial valuations over a $d$-dimensional variety are characterized among nontrivial valuations 
by having residue fields of
maximal transcendence degree $d-1$, this implies that $\trdeg_k(k_{v_\a}) \ge \trdeg_k(k_{v_\b}) = d-1$ and
thus that $v_\a$ is divisorial. 

We prove the claim by induction on the embedding dimension of $\O_{Y_\infty,\b}$, 
which is finite by \cref{t:stable-points-finite-embdim}. For this,
we rely on the fact that $Y$ is smooth at $\b(0)$. 
First note that the claim is clear if $\b(0)$ has codimension 1 in $Y$, 
as this is the center of $v_\b$ in $Y$.
Assume therefore that $\b(0)$ has codimension $\ge 2$ in $Y$. 
Let $g \colon Y' \to Y$ be the blow-up of the closure $B$ of $\b(0)$,
and let $\b' \in Y'_\infty$ denote the lift of $\b$. 
Since $B$ is smooth at $\b(0)$, there is an open neighborhood $U \subset Y$ of $\b(0)$
such that $B \cap U$ is smooth. 
Then $g^{-1}(U) \subset Y'$ is the blow-up of a smooth subscheme of $U$, hence it is smooth, 
and since $g(\b'(0)) = \b(0) \in U$, we have $\b'(0) \in g^{-1}(U)$
and $Y'$ is smooth at $\b'(0)$. 
Note that $g$ is ramified at $\b'(0)$, since the fiber of $g$
over $\b(0)$ is positive dimensional. 
Then it follows by \cite[Theorem~9.2]{dFD20} (equivalently, one can use 
part~(2) of \cref{c:embdim-smooth-unramified}) 
that $\embdim(\O_{Y'_\infty,\b'}) < \embdim(\O_{Y_\infty,\b})$, and induction applies. 
This settles assertion \eqref{i1:stable=divisorial+bounds}.

We now address \eqref{i2:stable=divisorial+bounds}.
Let $h \colon X' \to X$ be a proper birational morphism where $X'$ is a normal variety
with a prime divisor $E$ such that $v = q\ord_E$ for some $q \ge 1$. 
We denote by $\a'$ and $\a_v'$ the lifts of $\a$ and $\a_v$ to $X'_\infty$.
Note that $\a'(0) = \a'_v(0)$ is the generic point of $E$, which is contained in the smooth
locus of $X'$. In particular, the local rings $\O_{X'_\infty,\a'}$ and $\O_{X'_\infty,\a'_v}$ are regular, 
with the latter of dimension $q$.

It was observed in \cite[Remark~3.5]{MR18} that, in this setting,
\begin{equation}
\label{eq:MR-3.5}
\embdim(\O_{X_\infty,\a}) = \dim(\O_{X'_\infty,\a'}) + v(\Jac_h).
\end{equation}
In \cite{MR18} the characteristic is assumed to be zero, but this is not necessary;
in fact, since $\dim(\O_{X'_\infty,\a'}) = \embdim(\O_{X'_\infty,\a'})$, 
the formula also follows directly 
from part \eqref{i2:embdim-smooth-unramified} of \cref{c:embdim-smooth-unramified}, 
which does not require any assumption on characteristic. 
Now, using that
$\dim(\O_{X'_\infty,\a'}) = \codim(\a',\a'_v) + q$ (which holds because $q = \dim(\O_{X'_\infty,\a'_v})$)
and
$v(\Jac_h) = \embdim(\O_{X_\infty,\a_v}) - q$ (which follows by \cref{t:edim=Mather}),
we deduced from \eqref{eq:MR-3.5} that
\[
\embdim(\O_{X_\infty,\a}) = \codim(\a',\a'_v) + \embdim(\O_{X_\infty,\a_v}).
\]
As $h_\infty$ is injective on $\Spec \O_{X'_\infty,\a'}$, we have that
$\codim(\a',\a'_v) \le \codim(\a,\a_v)$, hence we end up with the inequality
\[
\embdim(\O_{X_\infty,\a}) \le \codim(\a,\a_v) + \embdim(\O_{X_\infty,\a_v}).
\]
On the other hand, \cref{t:semicont} (along with \cref{t:edim=Mather})
yields the opposite inequality
\[
\embdim(\O_{X_\infty,\a}) \ge \codim(\a,\a_v) + \embdim(\O_{X_\infty,\a_v}),
\]
hence the first equality stated in \eqref{eq:bound-Mather} holds, and
with \cref{t:edim=Mather} already giving $\embdim(\O_{X_\infty,\a_v}) = \^a_{v}(X)$, 
this settles \eqref{eq:bound-Mather}.
Note that the argument also shows that $\codim(\a',\a'_v) = \codim(\a,\a_v)$, 
and since any proper birational model $X' \to X$ can be dominated by
a model as the one used above, this implies 
the last assertion of the theorem.

Finally, the formula in \eqref{eq:bound-MJ} is proved using the same argument as in \cite[Theorem~11.1]{CdFD}, 
by combining \eqref{eq:bound-Mather} and \cref{t:embcodim-bound} with \cite[Proposition~7.4]{CdFD}.
\end{proof}

\begin{remark}
As discussed in the proof, one inequality in the formula stated in \eqref{eq:bound-Mather} 
already follows from \cite[Remark~3.5]{MR18} (at least when $\charK k = 0$), 
and our main contribution is to establish the opposite inequality. 
The inequality stated in \eqref{eq:bound-MJ} was proved in \cite[Theorem 4.1]{MR18} for maximal divisorial
arcs (i.e., when $\a = \a_v$) assuming $\charK k = 0$, 
and our result extends Mourtada and Reguera's theorem in two way, by removing the assumption on the characteristic
and by including all stable points in the statement. 
\end{remark}

We close this section by mentioning the follow semicontinuity property of Mather log discrepancies. 
The property already follows from \cite[Theorem~3.8]{dFEI08} (if $\charK k = 0$)
and \cite[Corollary~11.6]{dFD20} (in general), and is probably known to experts;
a more precise formulation when $X = \A^2_k$ appears for instance in \cite[Corollary 4.19]{FdBPPPP17}. 
We give here a different proof
based on the semicontinuity property of embedding dimension discussed earlier. 

\begin{corollary}
\label{t:semicont-^A-arc-top}
Let $X$ be a variety over a perfect field $k$.
Let $v$ and $v'$ be two divisorial valuations on $X$ and
$\a_v,\a_{v'} \in X_\infty$ be the associated maximal divisorial arcs. 
If $\a_v$ is a specialization of $\a_{v'}$ then $\^a_v(X) \ge \^a_{v'}(X)$, 
and the inequality is strict if $v \ne v'$.
\end{corollary}

\begin{proof}
The assertion follows from \cref{t:edim=Mather} and the fact that, by 
\cref{t:semicont}, if $\a_{v}$ is a specialization of $\a_{v'}$ then 
$\embdim(\O_{X_\infty,\a_v}) \ge \embdim(\O_{X_\infty,\a_{v'}})$
and the inequality is strict if $\a_v \ne \a_{v'}$. 
\end{proof}

\begin{remark}
The condition in \cref{t:semicont-^A-arc-top} that $\a_v$ is a specialization of $\a_{v'}$
can be seen as defining a notion of specialization among divisorial valuations.
The upper topology this defines on the set $\DivVal_X^\Z$ of divisorial valuations on $X$
agrees with the topology induced from the Zariski topology of $X_\infty$ via the inclusion map
$\DivVal_X^\Z \inj X_\infty$ defined by mapping a divisorial valuation $v$
to the associated maximal divisorial arc $a_v$.
There are of course other natural ways to equip $\DivVal_X^\Z$ with a topology.
For instance, $\DivVal_X^\Z$ is a subset
of the valuation space $\Val_X$ parameterizing
real valuations of the function field of $X$, 
and hence of the Berkovich analytification $X^\an$ of $X$
over $k$ regarded as a trivially valued field. 
These spaces have interesting topologies, 
but the induced topology on $\DivVal_X^\Z$ is discrete, 
which is not particularly interesting for our purpose. 
Similarly one can look at the 
map from $\DivVal_X^\Z$ to the Riemann--Zariski space $\RZ(X)$ 
sending a divisorial valuation $v$ to the corresponding
valuation ring $\O_v$. The fibers of this map are given by rescaling valuations, 
and pulling back the Zariski topology of $\RZ(X)$ again
does not produce any interesting notion of specialization. 
\end{remark}

Questions about semicontinuity properties of log discrepancies are certainly not new. 
Shokurov and Ambro's conjectures on the semicontinuity 
of minimal log discrepancies \cite{Sho88,Amb99} and their solutions in the
locally complete intersection case \cite{EMY03,EM04} immediately come to mind, 
and so does the semicontinuity property of minimal 
Mather--Jacobian log discrepancies on arbitrary varieties \cite{Ish13,dFD14}.
It should be stressed, however, that while these properties are
about lower-semicontinuity,
\cref{t:semicont-^A-arc-top} gives an upper-semicontinuity
property of Mather log discrepancies, 
which is more in line with Ishii's lower-bound on this invariant \cite[Theorem~1.1]{Ish13}. 
One should also be aware that \cref{t:semicont-^A-arc-top}
does not imply -- at least in an obvious way -- that minimal Mather log discrepancies
are upper-semicontinuous on $X$.


\end{document}